\pdfoutput=1
\documentclass[11pt,a4paper]{article}
\usepackage[T1]{fontenc}
\usepackage{authblk}

\usepackage{color}
\usepackage{helvet}         
\usepackage{courier}        
\usepackage{type1cm}        
\usepackage{makeidx}         
\usepackage{graphicx}        
\usepackage{multicol}        
\usepackage[bottom]{footmisc}
\usepackage{amsmath}
\usepackage{amssymb}
\usepackage{bbold}
\usepackage{amsthm}
\usepackage{subcaption}
\usepackage{comment}
\newtheorem{theorem}{Theorem}
\newtheorem{corollary}[theorem]{Corollary}

\newtheorem{lemma}[theorem]{Lemma}

\newtheorem{proposition}[theorem]{Proposition}

\newtheorem{remark}[theorem]{Remark}
\newtheorem{definition}[theorem]{Definition}
\usepackage[top=2cm, bottom=2cm, left=2cm, right=2cm]{geometry}
\numberwithin{theorem}{section}
\numberwithin{figure}{section}
\numberwithin{equation}{section}

\DeclareMathOperator{\dist}{dist}
\DeclareMathOperator{\SLE}{SLE}
\DeclareMathOperator{\hSLE}{hSLE}

\DeclareMathOperator{\GFF}{GFF}

\DeclareMathOperator{\free}{free}
\DeclareMathOperator{\hF}{{}_2F_1}
\DeclareMathOperator{\simple}{simple}

\begin{document}

\title{Convergence of the Critical Planar Ising Interfaces to Hypergeometric SLE}
\author{Hao Wu\thanks{hao.wu.proba@gmail.com}\\
NCCR/SwissMAP, Section de Math\'{e}matiques, Universit\'{e} de Gen\`{e}ve, Switzerland\\
\textit{and} Yau Mathematical Sciences Center, Tsinghua University, China}
\date{}
\affil{}

%
%
\maketitle

\abstract{We consider the planar Ising model in rectangle $(\Omega; x^L, x^R, y^R, y^L)$ with alternating boundary condition: $\ominus$ along $(x^Lx^R)$ and $(y^Ry^L)$, $\xi^R\in\{\oplus, \free\}$ along $(x^Ry^R)$, and $\xi^L\in\{\oplus, \free\}$ along $(y^Lx^L)$. We prove that the interface of critical Ising model with these boundary conditions converges to the so-called hypergeometric SLE$_3$. The method developed in this paper does not require constructing new holomorphic observable and the input is the convergence of the interface with Dobrushin boundary condition. This method could be applied to other lattice models, for instance Loop-Erased Random Walk and level lines of discrete Gaussian Free Field.}
\newcommand{\eps}{\epsilon}
\newcommand{\ov}{\overline}
\newcommand{\U}{\mathbb{U}}
\newcommand{\T}{\mathbb{T}}
\newcommand{\HH}{\mathbb{H}}
\newcommand{\LA}{\mathcal{A}}
\newcommand{\LB}{\mathcal{B}}
\newcommand{\LC}{\mathcal{C}}
\newcommand{\LD}{\mathcal{D}}
\newcommand{\LF}{\mathcal{F}}
\newcommand{\LK}{\mathcal{K}}
\newcommand{\LE}{\mathcal{E}}
\newcommand{\LG}{\mathcal{G}}
\newcommand{\LL}{\mathcal{L}}
\newcommand{\LM}{\mathcal{M}}
\newcommand{\LQ}{\mathcal{Q}}
\newcommand{\LP}{\mathcal{P}}
\newcommand{\LU}{\mathcal{U}}
\newcommand{\LV}{\mathcal{V}}
\newcommand{\LZ}{\mathcal{Z}}
\newcommand{\LH}{\mathcal{H}}
\newcommand{\R}{\mathbb{R}}
\newcommand{\C}{\mathbb{C}}
\newcommand{\N}{\mathbb{N}}
\newcommand{\Z}{\mathbb{Z}}
\newcommand{\E}{\mathbb{E}}
\newcommand{\PP}{\mathbb{P}}
\newcommand{\QQ}{\mathbb{Q}}
\newcommand{\A}{\mathbb{A}}
\newcommand{\one}{\mathbb{1}}
\newcommand{\bn}{\mathbf{n}}
\newcommand{\MR}{MR}
\newcommand{\cond}{\,|\,}
\newcommand{\la}{\langle}
\newcommand{\ra}{\rangle}
\newcommand{\tree}{\Upsilon}

\section{Introduction}
The Lenz-Ising model is introduced to model the ferromagnetism in statistical mechanics.  Due to celebrated work of Chelkak and Smirnov \cite{ChelkakSmirnovIsing}, it is proved that at the critical temperature, the interface of Ising model is conformally invariant. In particular, the interface of critical Ising model with Dobrushin boundary condition converges to $\SLE_3$ \cite{CDCHKSConvergenceIsingSLE}, and the interface of critical Ising model with free boundary condition converges to $\SLE_3(-3/2;-3/2)$ \cite{HonglerKytolaIsingFree, IzyurovObservableFree}, and the interface for multiply-connected domains \cite{IzyurovIsingMultiplyConnectedDomains}.
In these cases, the proofs are based on constructing holomorphic observables. In this paper, we study the scaling limit of the critical Ising model with alternating boundary conditions: we consider critical Ising model in a rectangle $(\Omega; x^L, x^R, y^R, y^L)$ with $\ominus$ along $(x^Lx^R)$ and $(y^Ry^L)$, $\xi^R\in\{\oplus, \free\}$ along $(x^Ry^R)$, and $\xi^L\in\{\oplus, \free\}$ along $(y^Lx^L)$. With these boundary conditions, on the event that there is a vertical crossing of $\ominus$, we see that there are two interfaces in the model $(\eta^L; \eta^R)$ where $\eta^L$ is an interface from $x^L$ to $y^L$ and $\eta^R$ is an interface from $x^R$ to $y^R$. In this paper, we study the law of the pair $(\eta^L;\eta^R)$. The scaling limit of $\eta^L$ is the so-called hypergeometric SLE, denoted by $\hSLE$. 

There are two features on the method developed in this paper. First, constructing holomorphic observable is the usual way to prove the convergence of interfaces in the critical lattice model; however, with our method, there is no need to construct new observable. The only input we need is the convergence of the interface with Dobrushin boundary condition.  Second, there are many works on multiple SLEs trying to study the scaling limit of interfaces in critical lattice model with alternating boundary conditions, see \cite{DubedatCommutationSLE, KytolaMultipleSLE, KytolaPeltolaMultipleSLE}, and their works study the local growth of these interfaces. Whereas, our result is ``global": we prove that the scaling limit of $\eta^L$ is $\hSLE_3$ as a continuous curve from $x^L$ to $y^L$. Moreover, the method developed in this paper also works for other lattice models as long as we have the convergence with Doburshin boundary conditions, for instance Loop-Erased Random Walk and level lines of discrete Gaussian Free Field. 

In this paper, we first study the properties of $\hSLE$ in Theorem \ref{thm::hyperSLE_reversibility}. Then we study the possible scaling limit of the pair of the interfaces $(\eta^L;\eta^R)$. We realize that there exists only one possible candidate for the limit of the pair $(\eta^L;\eta^R)$, see Theorem \ref{thm::slepair_lessthan4}. By identifying the only possible candidate, we prove the convergence of the pair of the interfaces $(\eta^L;\eta^R)$ in the critical Ising model with alternating boundary condition in Theorem \ref{thm::ising_cvg_pair}. In particular, this gives the convergence of $\eta^L$ to $\hSLE_3$.

\begin{theorem} \label{thm::hyperSLE_reversibility}
Fix $\kappa\in (0,8), \rho>(-4)\vee(\kappa/2-6)$, and $0<x<y$. Let $\eta$ be the $\hSLE_{\kappa}(\rho)$ in $\HH$ from $0$ to $\infty$ with marked points $(x,y)$. The process $\eta$ is almost surely generated by a continuous transient curve. Moreover, the process $\eta$ enjoys reversibility for $\rho\ge\kappa/2-4$: the time reversal of $\eta$ is the $\hSLE_{\kappa}(\rho)$ in $\HH$ from $\infty$ to $0$ with marked points $(y, x)$.  
\end{theorem}

\begin{theorem}\label{thm::slepair_lessthan4}
Fix a topological rectangle $(\Omega; x^L, x^R, y^R, y^L)$. Let $X_0(\Omega; x^L, x^R, y^R, y^L)$ be the collection of pairs of continuous curves $(\eta^L;\eta^R)$ in $\Omega$ such that $\eta^L$ (resp. $\eta^R$) is a continuous curve from $x^L$ to $y^L$ (resp. from $x^R$ to $y^R$) that does not intersect $(x^Ry^R)$ (resp. does not intersect $(y^Lx^L)$) and that $\eta^L$ is to the left of $\eta^R$.
Fix $\kappa\in (0,4]$ and $\rho^L>-2, \rho^R>-2$. 
\begin{itemize}
\item (Existence and Uniqueness) There exists a unique probability measure on $X_0(\Omega; x^L, x^R, y^R, y^L)$ with the following property: the conditional law of $\eta^R$ given $\eta^L$ is $\SLE_{\kappa}(\rho^R)$ with force point $x^R_+$ in the connected component of $\Omega\setminus\eta^L$ with $(x^Ry^R)$ on the boundary, and the conditional law of $\eta^L$ given $\eta^R$ is $\SLE_{\kappa}(\rho^L)$ with force point $x^L_-$ in the connected component of $\Omega\setminus\eta^R$ with $(y^Lx^L)$ on the boundary.
\item (Identification) Under this probability measure and fix $\rho^L=0$ and $\rho^R>-2$, the marginal law of $\eta^L$ is $\hSLE_{\kappa}(\rho^R)$ with marked points $(x^R, y^R)$.  
\end{itemize}
\end{theorem}

\begin{figure}[ht!]
\begin{center}
\includegraphics[width=0.8\textwidth]{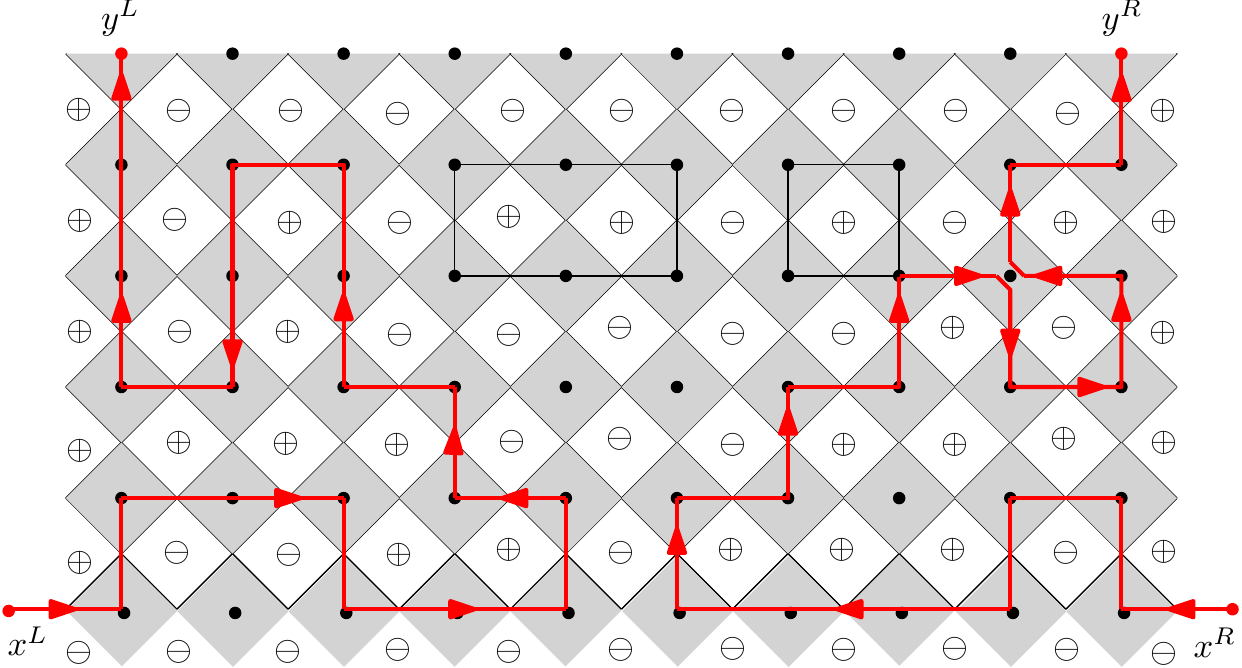}
\end{center}
\caption{\label{fig::ising_pair} The Ising interface with alternating boundary condition. }
\end{figure}

\begin{theorem}\label{thm::ising_cvg_pair}
Let discrete domains $(\Omega_{\delta}; x^L_{\delta}, x^R_{\delta}, y^R_{\delta}, y^L_{\delta})$ on the square lattice approximate some topological rectangle $(\Omega; x^L, x^R, y^R, y^L)$ as $\delta\to 0$. Consider the critical Ising model in $\Omega_{\delta}$ with the alternating boundary condition: $\ominus$ on $(x^L_{\delta}x^R_{\delta})$ and $(y^R_{\delta}y^L_{\delta})$, $\xi^R\in\{\oplus, \free\}$ on $(x^R_{\delta}y^R_{\delta})$, and $\xi^L\in\{\oplus, \free\}$ on $(y^L_{\delta}x^L_{\delta})$. Conditioned on the event that there exists a vertical crossing of $\ominus$, then there exists a pair of interfaces $(\eta^L_{\delta};\eta^R_{\delta})$ where $\eta^L_{\delta}$ (resp. $\eta^R_{\delta}$) is the interface connecting $x^L_{\delta}$ to $y^L_{\delta}$ (resp. connecting $x^R_{\delta}$ to $y^R_{\delta}$).  The law of the pair $(\eta^L_{\delta}; \eta^R_{\delta})$ converges weakly to the pair of $\SLE$ curves in Theorem \ref{thm::slepair_lessthan4} as $\delta\to 0$ where $\kappa=3$ and $\xi^R, \xi^L,\rho^R,\rho^L$ are related in the following way: for $q\in\{L, R\}$,
\[\rho^q=0,\quad\text{if }\xi^q=\oplus;\quad \rho^q=-3/2,\quad\text{if }\xi^q=\free.\]
In particular, when $\xi^R=\oplus$ and $\xi^L=\oplus$, 
the law of $\eta^L_{\delta}$ converges weakly to $\hSLE_{3}(0)$ as $\delta\to 0$; when $\xi^R=\oplus$ and $\xi^L=\free$, the law of $\eta^L_{\delta}$ converges weakly to $\hSLE_3(-3/2)$ as $\delta\to 0$.
\end{theorem}

A similar conclusion as in Theorem \ref{thm::slepair_lessthan4} also holds for multiple $\SLE$ curves; however, we can only identify the marginal law for the case $\kappa=4$, see Remark \ref{rem::slepair_lessthan4_degenerate}, and
it is difficult to identify the marginal law of the curves for general $\kappa$. Instead, we could derive the marginal law for the degenerate case: when all the starting points of curves coincide and all the ending points of curves coincide, the marginal law of the curves becomes $\SLE_{\kappa}(\rho)$ process. 

\begin{proposition}\label{prop::slemultiple_degenerate}
Fix a Dobrushin domain $(\Omega; x, y)$ and an integer $n\ge 2$. Let $X_0^n(\Omega; x, y)$ be the collection of $n$ non-intersecting curves $(\eta_1;...;\eta_n)$ where $\eta_j$ is a continuous curve in $\Omega$ from $x$ to $y$ for $j\in \{1,...,n\}$ and that $\eta_j$ is to the right of $\eta_{j-1}$ and is to the left of $\eta_{j+1}$ with the convention that $\eta_0=(yx)$ and $\eta_{n+1}=(xy)$. Fix $\kappa\in (0,4]$. 
\begin{itemize}
\item (Existence and Uniqueness) There exists a unique probability measure on $X_0^n(\Omega; x, y)$ with the following property: for each $j\in\{1,...,n\}$, the conditional law of $\eta_j$ given $\eta_{j-1}$ and $\eta_{j+1}$ is $\SLE_{\kappa}$ in the region between $\eta_{j-1}$ and $\eta_{j+1}$. 
\item (Identification) Under this probability measure, the marginal law of $\eta_j$ is $\SLE_{\kappa}(\rho^L_j;\rho^R_j)$ for $j\in \{1,...,n\}$ where $\rho^L_j=2j-2,\rho^R_j=2n-2j$.
\end{itemize} 
\end{proposition}

In Theorem \ref{thm::slepair_lessthan4} and Proposition \ref{prop::slemultiple_degenerate}, we focus on $\kappa\in (0,4]$. Readers may wonder whether we have similar conclusion for $\kappa\in (4,8)$. In fact, we believe the conclusion in Theorem \ref{thm::slepair_lessthan4} also holds for $\kappa\in (4,8)$. However, both of the two parts are unknown to our knowledge. Whereas, we can still show a weaker version of Theorem \ref{thm::slepair_lessthan4} for the degenerate case when $\kappa\in (4,8)$, see Lemma \ref{lem::slemultiple}. By applying this result to FK-Ising model, we have the following conclusion. 

\begin{proposition}\label{prop::fkising_cvg_pair}
Let discrete domains $(\Omega_{\delta}; x^L_{\delta}, x^R_{\delta}, y^R_{\delta}, y^L_{\delta})$ on the square lattice approximate some Dobrushin domain $(\Omega; x, y)$ such that $x^L_{\delta}, x^R_{\delta}\to x$ and $y^L_{\delta}, y^R_{\delta}\to y$ as $\delta\to 0$. Consider the critical FK-Ising model in $\Omega_{\delta}$ with alternating boundary condition: free on $(x_{\delta}^Lx_{\delta}^R)$ and $(y_{\delta}^Ry_{\delta}^L)$, and wired on $(x_{\delta}^Ry_{\delta}^R)$ and $(y_{\delta}^Lx_{\delta}^L)$.
Conditioned on the event that there are two disjoint vertical dual-crossings, then there exists a pair of interfaces $(\eta^L_{\delta};\eta^R_{\delta})$ where $\eta^L_{\delta}$ (resp. $\eta^R_{\delta}$) is the interface connecting $x^L_{\delta}$ to $y^L_{\delta}$ (resp. connecting $x^R_{\delta}$ to $y^R_{\delta}$).
The law of $(\eta^L_{\delta}; \eta^R_{\delta})$ converges weakly to the unique pair of curves $(\eta^L;\eta^R)$ in $X^2_0(\Omega; x, y)$ with the following property: Given $\eta^L$, the conditional law of $\eta^R$ is an $\SLE_{16/3}$ conditioned not to hit $\eta^L$ except at the end points; given $\eta^R$, the conditional law of $\eta^L$ is an $\SLE_{16/3}$ conditioned not to hit $\eta^R$ except at the end points. In particular, the law of $\eta^L_{\delta}$ converges weakly to $\SLE_{\kappa}(\kappa-2)$ with $\kappa=16/3$ as $\delta\to 0$. 
\end{proposition}

To end the introduction, we summarize the relation between $\hSLE_{\kappa}(\rho)$ and $\SLE_{\kappa}(\rho)$ process:
\begin{itemize}
\item When $\rho=-2$, $\hSLE_{\kappa}(\rho)$ process is the same as $\SLE_{\kappa}$.
\item When $\kappa=4$, $\hSLE_4(\rho)$ process is the same as $\SLE_4(\rho+2, -\rho-2)$ with force points $(x,y)$.
\item When $y\to \infty$, $\hSLE_{\kappa}(\rho)$ process degenerates to $\SLE_{\kappa}(\rho+2)$ with force point $x$.
\end{itemize}

\smallbreak
\noindent\textbf{Outline and relation to previous work.} 
We will introduce hypergeometric SLE in Section \ref{sec::hyperSLE}. 
There were various papers working on variants of hypergeometric SLE with different motivations, see \cite{ZhanReversibility, QianConformalRestrictionTrichordal}. The definitions may be different from ours. There are some technicalities that arise when introducing hypergeometric SLE that do not have been fully addressed previously. We include a self-contained introduction to $\hSLE$ with our motivation in Section \ref{sec::hyperSLE}, treat the technical difficulties and show Theorem \ref{thm::hyperSLE_reversibility}. 
We will prove Theorem \ref{thm::slepair_lessthan4} and Proposition \ref{prop::slemultiple_degenerate} in Section \ref{sec::slepair}. 
The uniqueness part in Theorem \ref{thm::slepair_lessthan4} was proved in \cite[Theorem 4.1]{MillerSheffieldIG2} and the existence part when $\rho^L=\rho^R=0$ was proved in \cite{LawlerPartitionFunctionsSLE, KozdronLawlerMultipleSLEs}. 
We will introduce Ising model in Section \ref{sec::ising} and prove Theomem \ref{thm::ising_cvg_pair}. We will introduce FK-Ising model in Section \ref{sec::fkising} and prove Proposition \ref{prop::fkising_cvg_pair}. 
In \cite{IzyurovIsingMultiplyConnectedDomains}, the author proved ``local" convergence of Ising interfaces to $\hSLE_3$ by constructing holomorphic observables. We will show Theorem \ref{thm::ising_cvg_pair} by Theorem \ref{thm::slepair_lessthan4} (without constructing any new observable). Our approach is ``global", the
method only requires the input of convergence with Dobrushin boundary condition and it also works for other lattice models.

\section{Preliminaries}
\subsection{Space of curves}
\label{subsec::space_curves}

A planar curve is a continuous mapping from $[0,1]$ to $\C$ modulo reparameterization. Let $X$ be the set of planar curves. The metric $d$ on $X$ is defined by 
\[d(\eta_1,\eta_2)=\inf_{\varphi_1,\varphi_2}\sup_{t\in[0,1]}|\eta_1(\varphi_1(t))-\eta_2(\varphi_2(t))|,\]
where the inf is over increasing homeomorphisms $\varphi_1,\varphi_2:[0,1]\to [0,1]$. The metric space $(X, d)$ is complete and separable. A simple curve is a continuous injective mapping from $[0,1]$ to $\C$ modulo reparameterization. 
Let $X_{\simple}$ be the subspace of simple curves and denote by $X_0$ its closure. The curves in $X_0$ may have multiple points but they do not have self-crossings. 

We call $(\Omega; x,y)$ a \textit{Dobrushin domain} if $\Omega$ is a non-empty simply connected proper subset of $\C$ and $x,y$ are two distinct boundary points. Denote by $(xy)$ the arc of $\partial\Omega$ from $x$ to $y$ counterclockwise. We say that a sequence of Dobrushin domains $(\Omega_{\delta}; a_{\delta}, b_{\delta})$ converges to a Dobrushin domain $(\Omega; a, b)$ in the \textit{Carath\'eodory sense} if $f_{\delta}\to f$ uniformly on any compact subset of $\HH$ where $f_{\delta}$ (resp. $f$) is the unique conformal map from $\HH$ to $\Omega_{\delta}$ (resp. $\Omega$) satisfying $f_{\delta}(0)=a_{\delta}, f_{\delta}(\infty)=b_{\delta}$ and $f_{\delta}'(\infty)=1$ (resp. $f(0)=a, f(\infty)=b, f'(\infty)=1$). 

Given a Dobrushin domain $(\Omega; x,y)$, let $X_{\simple}(\Omega; x,y)$ be the space of simple curves $\eta$ such that \[\eta(0)=x, \quad\eta(1)=y,\quad\eta(0,1)\subset\Omega.\]
Denote by $X_0(\Omega; x,y)$ the closure of $X_{\simple}(\Omega; x,y)$. 

We call $(\Omega; a, b, c, d)$ a \textit{quad (or topological rectangle)} if $\Omega$ a non-empty simply connected proper subset of $\C$ and $a, b, c, d$ are four distinct boundary points in counterclockwise order. Given a quad $(\Omega; a, b, c, d)$, we denote by $d_{\Omega}((ab), (cd))$ the extremal distance between $(ab)$ and $(cd)$ in $\Omega$. We say a sequence of quads $(\Omega_{\delta}; a_{\delta}, b_{\delta}, c_{\delta}, d_{\delta})$ converges to a quad $(\Omega; a, b, c, d)$ in the \textit{Carath\'{e}odory sense} if $f_{\delta}\to f$ uniformly on any compact subset of $\HH$ and  $\lim_{\delta}f_{\delta}^{-1}(b_{\delta})=f^{-1}(b)$ and $\lim_{\delta}f_{\delta}^{-1}(c_{\delta})=f^{-1}(c)$ 
where $f_{\delta}$ (resp. $f$) is the unique conformal map from $\HH$ to $\Omega_{\delta}$ (resp. $\Omega$) satisfying $f_{\delta}(0)=a_{\delta}, f_{\delta}(\infty)=d_{\delta}$ and $f_{\delta}'(\infty)=1$ (resp. $f(0)=a, f(\infty)=d, f'(\infty)=1$). 

Given a quad $(\Omega; x^L, x^R, y^R, y^L)$, let $X_{\simple}(\Omega; x^L, x^R, y^R, y^L)$ be the collection of pairs of simple curves $(\eta^L;\eta^R)$ such that $\eta^L\in X_{\simple}(\Omega; x^L, y^L)$ and $\eta^R\in X_{\simple}(\Omega; x^R, y^R)$ and that $\eta^L\cap\eta^R=\emptyset$. The definition of $X_0(\Omega; x^L, x^R, y^R, y^L)$ is a little bit complicate. 
Given a quad $(\Omega; x^L, x^R, y^R, y^L)$ and $\eps>0$, let $X_0^{\eps}(\Omega; x^L, x^R, y^R, y^L)$ be the set of pairs of curves $(\eta^L; \eta^R)$ such that 
\begin{itemize}
\item $\eta^L\in X_0(\Omega; x^L, y^L)$ and $\eta^R\in X_0(\Omega; x^R, y^R)$;
\item $d_{\Omega^L}(\eta^L, (x^Ry^R))\ge \eps$ where $\Omega^L$ the connected component of $\Omega\setminus \eta^L$ with $(x^Ry^R)$ on the boundary, 
and $\eta^R$ is contained in the closure of $\Omega^L$ ;
\item $d_{\Omega^R}(\eta^R, (y^Lx^L))\ge \eps$ where $\Omega^R$ the connected component of  $\Omega\setminus \eta^R$ with $(y^Lx^L)$ on the boundary, and $\eta^L$ is contained in the closure of $\Omega^R$.
\end{itemize}
Define the metric on $X_0^{\eps}(\Omega; x^L, x^R, y^R, y^L)$ by 
\[\LD((\eta^L_1, \eta^R_1), (\eta^L_2,\eta^R_2))=\max\{d(\eta^L_1,\eta^L_2), d(\eta^R_1,\eta^R_2)\}. \]
One can check $\LD$ is a metric and the space $X_0^{\eps}(\Omega; x^L, x^R, y^R, y^L)$ with $\LD$ is complete and separable. 
Finally, set 
\[X_0(\Omega; x^L, x^R, y^R, y^L)=\bigcup_{\eps>0}X_0^{\eps}(\Omega; x^L, x^R, y^R, y^L).\]
Note that $X_0(\Omega; x^L, x^R, y^R, y^L)$ is no longer complete. 
\smallbreak
Suppose $E$ is a metric space and $\LB_E$ is the Borel $\sigma$-field. Let $\LP$ be the space of probability measures on $(E, \LB_E)$. The Prohorov metric $d_{\LP}$ on $\LP$ is defined by  
\[d_{\LP}(\PP_1, \PP_2)=\inf\left\{\eps>0: \PP_1[A]\le \PP_2[A^{\eps}]+\eps, \PP_2[A]\le \PP_1[A^{\eps}]+\eps, \forall A\in\LB_E\right\}.\]
When $E$ is complete and separable, the space $\LP$ is complete and separable (\cite[Theorem 6.8]{BillingsleyConvergenceProbabilityMeasures}); moreover, a sequence $\PP_n$ in $\LP$ converges weakly to $\PP$ if and only if $d_{\LP}(\PP_n, \PP)\to 0$. 

Let $\Sigma$ be a family of probability measures on $(E,\LB_E)$. We call $\Sigma$ \textit{relatively compact} if every sequence of elements in $\Sigma$ contains a weakly convergent subsequence. We call $\Sigma$ \textit{tight} if, for every $\eps>0$, there exists a compact set $K_{\eps}$ such that $\PP[K_{\eps}]\ge 1-\eps$ for all $\PP\in\Sigma$. By Prohorov's Theorem (\cite[Theorem 5.2]{BillingsleyConvergenceProbabilityMeasures}), when $E$ is complete and separable, relative compactness is equivalent to tightness.

\subsection{Loewner chain}

We call a compact subset $K$ of $\overline{\HH}$ an \textit{$\HH$-hull} if $\HH\setminus K$ is simply connected. Riemann's Mapping Theorem asserts that there exists a unique conformal map $g_K$ from $\HH\setminus K$ onto $\HH$ such that
$\lim_{z\to\infty}|g_K(z)-z|=0$.
We call such $g_K$ the conformal map from $\HH\setminus K$ onto $\HH$ normalized at $\infty$ and we call $a(K):=\lim_{z\to \infty} z(g_t(z)-z)$ the \textit{half-plane capacity} of $K$. 

\textit{Loewner chain} is a collection of $\HH$-hulls $(K_{t}, t\ge 0)$ associated with the family of conformal maps $(g_{t}, t\ge 0)$ obtained by solving the Loewner equation: for each $z\in\mathbb{H}$,
\[\partial_{t}{g}_{t}(z)=\frac{2}{g_{t}(z)-W_{t}}, \quad g_{0}(z)=z,\]
where $(W_t, t\ge 0)$ is a one-dimensional continuous function which we call the driving function. Let $T_z$ be the \textit{swallowing time} of $z$ defined as $\sup\{t\ge 0: \min_{s\in[0,t]}|g_{s}(z)-W_{s}|>0\}$.
Let $K_{t}:=\overline{\{z\in\mathbb{H}: T_{z}\le t\}}$. Then $g_{t}$ is the unique conformal map from $H_{t}:=\mathbb{H}\backslash K_{t}$ onto $\mathbb{H}$ normalized at $\infty$. Since the half-plane capacity for $K_t$ is $2t$ for all $t\ge 0$, we say that the process $(K_t, t\ge 0)$ is parameterized by the half-plane capacity. We say that $(K_t, t\ge 0)$ can be generated by the continuous curve $(\eta(t), t\ge 0)$ if for any $t$, the unbounded connected component of $\HH\setminus\eta[0,t]$ coincides with $H_t=\HH\setminus K_t$. 

Here we discuss about the evolution of a point $y\in\R$ under $g_t$. We assume $y\ge 0$. There are two possibilities: if $y$ is not swallowed by $K_t$, then we define $Y_t=g_t(y)$; if $y$ is swallowed by $K_t$, then we define $Y_t$ to be the image of the rightmost of point of $K_t\cap\R$ under $g_t$. Suppose that $(K_t, t\ge 0)$ is generated by a continuous path $(\eta(t), t\ge 0)$ and that the Lebesgue measure of $\eta[0,\infty]\cap\R$ is zero. Then the process $Y_t$ is uniquely characterized by the following equation: 
\[Y_t=y+\int_0^t \frac{2ds}{Y_s-W_s},\quad Y_t\ge W_t,\quad \forall t\ge 0.\] 
In this paper, we may write $g_t(y)$ for the process $Y_t$. 

The convention of driving function can be defined for any simply connected domain via conformal transformation. 
Fix a Dobrushin domain $(\Omega; x, y)$ and let $\phi$ be some fixed conformal map from $\Omega$ onto $\HH$ such that $\phi(x)=0$ and $\phi(y)=\infty$. Suppose $\eta\in X_{\simple}(\Omega; x, y)$.  Then $\phi(\eta)$ is a continuous curve in $X_{\simple}(\HH; 0,\infty)$. Thus, if we parameterize $\phi(\eta)$ by its half-plane capacity, then it has a continuous driving process. We use the term the driving process of $\eta$ in $\Omega$ to indicate the driving process in half-plane capacity in $\HH$ after the transformation $\phi$. 

\subsection{Convergence of curves}

In this section, we first recall the main result of \cite{KemppainenSmirnovRandomCurves} and then show a similar result for pairs of curves.  Suppose $(Q; a, b, c, d)$ is a quad. 
We say that a curve $\eta$ \textit{crosses} $Q$ if there exists a subinterval $[s,t]$ such that $\eta(s,t)\subset Q$ and $\eta[s,t]$ intersects both $(ab)$ and $(cd)$. 
Fix a Dobrushin domain $(\Omega; x, y)$, 
for any curve $\eta$ in $X_0(\Omega; x,y)$ and any time $\tau$,  define $\Omega_{\tau}$ to be the connected component of $\Omega\setminus\eta[0,\tau]$ with $y$ on the boundary. 
Consider a quad $(Q, a, b, c, d)$ in $\Omega_{\tau}$ such that $(bc)$ and $(da)$ are contained in $\partial \Omega_{\tau}$. We say that $Q$ is \textit{avoidable} if it does not disconnect $\eta(\tau)$ from $y$ in $\Omega_{\tau}$.    

\begin{definition}
A family $\Sigma$ of probability measures on curves in $X_{\simple}(\Omega; x, y)$ is said to satisfy \textbf{Condition C2} if, for any $\eps>0$, there exists a constant $c(\eps)>0$ such that for any $\PP\in\Sigma$, any stopping time $\tau$, and any avoidable quad $(Q; a, b, c, d)$ in $\Omega_{\tau}$ such that $d_{Q}((ab), (cd))\ge c(\eps)$, we have 
\[\PP[\eta[\tau,1]\text{ crosses }Q\cond \eta[0,\tau]]\le 1-\eps.\]
\end{definition}

\begin{theorem}\label{thm::cvg_curves_chordal}
\cite[Corollary 1.7, Proposition 2.6]{KemppainenSmirnovRandomCurves}. 
Fix a Dobrushin domain $(\Omega; x, y)$. 
Suppose that $(\eta_n)_{n\in\N}$ is a sequence of curves in $X_{\simple}(\Omega; x, y)$ satisfying Condition C2. Denote by  $(W_n(t), t\ge 0)$ the driving process of $\eta_n$. Then 
\begin{itemize}
\item the family $(W_n)_{n\in\N}$ is tight in the metrisable space of continuous functions on $[0,\infty)$ with the topology of uniform convergence on compact subsets of $[0,\infty)$;
\item the family $(\eta_n)_{n\in\N}$ is tight in the space of curves $X$;
\item the family $(\eta_n)_{n\in\N}$, when each curve is parameterized by the half-plane capacity, is tight in the metrisable space of continuous functions on $[0,\infty)$ with the topology of uniform convergence on compact subsets of $[0,\infty)$.
\end{itemize}
 Moreover, if the sequence converges in any of the topologies above it also converges in the two other topologies and
 the limits agree in the sense that the limiting random curve is driven by the limiting driving function.
\end{theorem}

Next, we will explain a similar result for pairs of curves. Fix a quad $(\Omega; x^L, x^R, y^R, y^L)$. 

\begin{definition}
A family $\Sigma$ of probability measures on pairs of curves in $X_{\simple}(\Omega; x^L, x^R, y^R, y^L)$ is said to satisfy \textbf{Condition C2} if, for any $\eps>0$, there exists a constant $c(\eps)>0$ such that for any $\PP\in\Sigma$, the following holds.
Given any $\eta^L$-stopping time $\tau^L$ and any $\eta^R$-stopping time $\tau^R$,  and any avoidable quad $(Q^R; a^R, b^R, c^R, d^R)$ for $\eta^R$ in $\Omega\setminus (\eta^L[0,\tau^L]\cup\eta^R[0,\tau^R])$ such that $d_{Q^R}((a^Rb^R), (c^Rd^R))\ge c(\eps)$, and any avoidable quad $(Q^L; a^L, b^L, c^L, d^L)$ for $\eta^L$ in $\Omega\setminus (\eta^L[0,\tau^L]\cup\eta^R[0,\tau^R])$ such that $d_{Q^L}((a^Lb^L), (c^Ld^L))\ge c(\eps)$, we have 
\[\PP\left[\eta^R[\tau^R, 1]\text{ crosses }Q^R\cond \eta^L[0,\tau^L], \eta^R[0,\tau^R]\right]\le 1-\eps,\]
\[\PP\left[\eta^L[\tau^L, 1]\text{ crosses }Q^L\cond \eta^L[0,\tau^L], \eta^R[0,\tau^R]\right]\le 1-\eps.\]
\end{definition}

\begin{theorem}\label{thm::cvg_pairs}
Suppose that $\{(\eta^L_n;\eta^R_n)\}_{n\in\N}$ is a sequence of pairs of curves in $X_{\simple}(\Omega; x^L, x^R, y^R, y^L)$ and denote their laws by $\{\PP_n\}_{n\in\N}$. Let $\Omega_n^L$ be the connected component of $\Omega\setminus\eta_n^L$ with $(x^Ry^R)$ on the boundary and $\Omega_n^R$ be the connected component of $\Omega\setminus\eta_n^R$ with $(y^Lx^L)$ on the boundary. 
Define, for each $n$, 
\[\LD^L_n=d_{\Omega_n^L}(\eta^L_n, (x^Ry^R)),\quad \LD^R_n=d_{\Omega_n^R}(\eta^R_n, (y^Lx^L)).\]
Assume that the family $\{(\eta^L_n; \eta^R_n)\}_{n\in\N}$ satisfies Condition C2 and that the sequence of random variables $\{(\LD^L_n; \LD^R_n)\}_{n\in\N}$ is tight in the following sense: for any $u>0$, there exists $\eps>0$ such that 
\[\PP_n\left[\LD^{L}_n\ge\eps, \LD^R_n\ge\eps\right]\ge 1-u,\quad \forall n.\]
Then the sequence $\{(\eta^L_n;\eta^R_n)\}_{n\in\N}$ is relatively compact in $X_0(\Omega; x^L, x^R, y^R, y^L)$. 
\end{theorem}
\begin{proof}
By Theorem \ref{thm::cvg_curves_chordal}, we know that there is subsequence $n_k\to\infty$ such that $\eta^L_{n_k}$ (resp. $\eta^R_{n_k}$) converges weakly in all three topologies in Theorem \ref{thm::cvg_curves_chordal}. By Skorohod Represnetation Theorem, we could couple all $(\eta^L_{n_k};\eta^R_{n_k})$ in a common space so that $\eta^L_{n_k}\to\eta^L$ and $\eta^R_{n_k}\to\eta^R$ almost surely. 
For $\eps>0$, define 
\[K_{\eps}=\left\{(\eta^L; \eta^R)\in X_{\simple}(\Omega; x^L, x^R, y^R, y^L): d_{\Omega^L}(\eta^L, (x^Ry^R))\ge\eps, d_{\Omega^R}(\eta^R, (y^Lx^L))\ge\eps\right\}.\]
From the assumption, we know that, for any $u>0$, there exists $\eps>0$ such that $\inf_n\PP_n[K_{\eps}]\ge 1-u$. 
Therefore, with probability at least $1-u$, the sequence $(\eta^L_{n_k}; \eta^R_{n_k})$ converges to $(\eta^L; \eta^R)$ in $X^{\eps}_0(\Omega; x^L, x^R, y^R, y^L)\subset X_0(\Omega; x^L, x^R, y^R, y^L)$. This is true for any $u>0$, thus we have $(\eta^L_{n_k};\eta^R_{n_k})$ converges to $(\eta^L;\eta^R)$ in $X_0(\Omega; x^L, x^R, y^R, y^L)$ almost surely. 
\end{proof}

\subsection{SLE and SLE$_{\kappa}(\rho)$}

\textit{Schramm Loewner Evolution} $\SLE_{\kappa}$ is the random Loewner chain $(K_{t}, t\ge 0)$ driven by $W_t=\sqrt{\kappa}B_t$ where $(B_t, t\ge 0)$ is a standard one-dimensional Brownian motion.
In \cite{RohdeSchrammSLEBasicProperty}, the authors prove that $(K_{t}, t\ge 0)$ is almost surely generated by a continuous transient curve, i.e. there almost surely exists a continuous curve $\eta$ such that for each $t\ge 0$, $H_{t}$ is the unbounded connected component of $\mathbb{H}\backslash\eta[0,t]$ and that $\lim_{t\to\infty}|\eta(t)|=\infty$. There are phase transitions at $\kappa=4$ and $\kappa=8$: $\SLE_{\kappa}$ are simple curves when $\kappa\in [0,4]$; they have self-touching when $\kappa\in (4,8)$; and they are space-filling when $\kappa\ge 8$. 

It is clear that $\SLE_{\kappa}$ is scaling invariant, thus we can define $\SLE_{\kappa}$ in any simply connected domain $D$ from one boundary point $x$ to another boundary point $y$ by the conformal image: let $\phi$ be a conformal map from $\HH$ onto $D$ that sends $0$ to $x$ and $\infty$ to $y$, then define $\phi(\eta)$ to be $\SLE_{\kappa}$ in $D$ from $x$ to $y$. 
For $\kappa\in (0,8)$, the curves $\SLE_{\kappa}$ enjoys \textit{reversibility}: let $\eta$ be an $
 \SLE_{\kappa}$ in $D$ from $x$ to $y$, then the time-reversal of $\eta$ has the same law as $\SLE_{\kappa}$ in $D$ from $y$ to $x$. The reversibility for $\kappa\in (0,4]$ was proved in \cite{ZhanReversibility}, and it was proved for $\kappa\in (4,8)$ in \cite{MillerSheffieldIG3}. 
 \smallbreak
 \textit{Hypergeometric functions} are defined for $a,b\in\R$, $c\in\R\setminus\{0,-1,-2,-3,...\}$ and for $|z|<1$ by 
 \[\hF(a,b,c;z)=\sum_{n=0}^{\infty}\frac{(a)_n(b)_n}{n!(c)_n}z^n,\]
where $(q)_n$ is the Pochhammer symbol defined by $(q)_n=1$ for $n=0$ and $(q)_n=q(q+1)\cdots(q+n-1)$ for $n\ge 1$. 
When $c>a+b$, denote by $\Gamma$ the Gamma function, we have (see \cite[Equation (15.1.20)]{AbramowitzHandbook}),
\[\hF(a,b,c;1)=\frac{\Gamma(c)\Gamma(c-a-b)}{\Gamma(c-a)\Gamma(c-b)}.\]
The hypergeometric function is a solution of Euler's hypergeometric differential equation:
\[z(1-z)F''(z)+(c-(a+b+1)z)F'(z)-abF(z)=0.\]
In this paper, we focus on $\kappa\in (0,8)$, $\rho\in\R$, and
\begin{equation}\label{eqn::hSLE_hyperF}
F(z):=\hF\left(\frac{2\rho+4}{\kappa}, 1-\frac{4}{\kappa}, \frac{2\rho+8}{\kappa};z\right).
\end{equation}
When $\kappa\in (0,8), \rho>(-4)\vee(\kappa/2-6)$, we have $F(1)\in (0,\infty)$. In particular, $F$ is smooth for $z\in (-1,1)$ and is continuous for $z\in (-1,1]$. 

\begin{lemma}\label{lem::hypersle_mart}
Fix $\kappa\in (0,8), \rho\in\R$, suppose $\eta$ is an $\SLE_{\kappa}$ in $\HH$ from 0 to $\infty$ and $(g_t,t\ge 0)$ is the corresponding family of conformal maps. Fix $0<x<y$ and let $T_x$ be the swallowing time of $x$. 
Define, for $t<T_x$,  
\[J_t=\frac{g_t'(x)g_t'(y)}{(g_t(y)-g_t(x))^{2}},\quad Z_t=\frac{g_t(x)-W_t}{g_t(y)-W_t}.\]
Let $F$ be defined through (\ref{eqn::hSLE_hyperF}). Then the following process is a local martingale:
\[M_t:=Z_t^a J_t^bF(Z_t)\one_{\{t<T_x\}},\quad\text{where }a=\frac{\rho+2}{\kappa}, \quad
b=\frac{(\rho+2)(\rho+6-\kappa)}{4\kappa}.\]
\end{lemma}
\begin{proof}
By It\^{o}'s formula, one can check
\[dJ_t=\frac{-2J_t}{(g_t(y)-W_t)^2}\left(\frac{1}{Z_t}-1\right)^2dt,\quad dZ_t=\frac{(Z_t-1)dW_t}{g_t(y)-W_t}+\frac{(1-Z_t)(2+(2-\kappa)Z_t)dt}{Z_t(g_t(y)-W_t)^2}.\]
Therefore, the process $J_t^b\psi(Z_t)\one_{\{t<T_x\}}$ is a local martingale if $\psi$ is twice-differentiable and satisfies
\[\kappa z^2(1-z)\psi''(z)+2z(2+(2-\kappa)z)\psi'(z)-4b(1-z)\psi=0.\]
One can check that $\psi(z):=z^{a}F(z)$ satisfies this ODE and hence $M_t$ is a local martingale. Moreover, we have
\[dM_t=M_t\left(\frac{a}{W_t-g_t(x)}+\frac{-a}{W_t-g_t(y)}-\frac{F'(Z_t)}{F(Z_t)}\left(\frac{1-Z_t}{g_t(y)-W_t}\right)\right)dW_t.\]
We will show that $M_t$ is actually a uniform integrable martingale when $\rho\ge \kappa/2-4$ in Proposition \ref{prop::hypersle_mart}. 
\end{proof}

\smallbreak

\textit{$\SLE_{\kappa}(\underline{\rho})$ processes} are variants of $\SLE_{\kappa}$ where one keeps track of multiple points on the boundary. Suppose $\underline{x}^L=(x^{l, L}<\cdots<x^{1,L})$ where $x^{1,L}\le 0$ and $\underline{x}^R=(x^{1,R}<\cdots<x^{r,R})$ where $x^{1,R}\ge 0$ and $\underline{\rho}^L=(\rho^{l,L},\cdots, \rho^{1,L}),\underline{\rho}^R=(\rho^{1,R},\cdots, \rho^{r,R})$ where $\rho^{i,q}\in\R$ for $q\in\{L, R\}$ and $i\in\N$. An $\SLE_{\kappa}(\underline{\rho}^L;\underline{\rho}^R)$ process with force points $(\underline{x}^L;\underline{x}^R)$ is the Loewner evolution driven by $W_t$ which is the solution to the system of integrated SDEs: 
\begin{align}\label{eqn::sle_rho_sde}
W_t&=\sqrt{\kappa}B_t+\sum_i\int_0^t\frac{\rho^{i,L} ds}{W_s-V_s^{i,L}}+\sum_i\int_0^t\frac{\rho^{i,R} ds}{W_s-V_s^{i,R}},\\
V^{i,q}_t&=x^{i,q}+\int_0^t\frac{2ds}{V^{i,q}_s-W_s},\quad \text{for }q\in\{L, R\}, i\in\N,\notag
\end{align}
where $B_t$ is one-dimensional Brownian motion. Define the \textit{continuation threshold} to be the infimum of the time $t$ for which 
\[\text{either}\quad \sum_{i: V^{i,L}_t=W_t}\rho^{i,L}\le -2,\quad \text{or}\quad \sum_{i: V^{i,R}_t=W_t}\rho^{i,R}\le -2.\]
The process is well-defined up to the continuation threshold, and it is generated by continuous curve up to and including the continuation threshold, see \cite{MillerSheffieldIG1}. 

In this paper, we only use $\SLE_{\kappa}(\underline{\rho})$ with two force points: $\SLE_{\kappa}(\rho^L;\rho^R)$ with force points $(x^L\le 0\le x^R)$ or $\SLE_{\kappa}(\rho,\nu)$ with force points $(0\le x<y)$. To simply notations, we only discuss properties of these two kinds of processes. The behavior of $\SLE_{\kappa}(\rho,\nu)$ varies according to different $\rho,\nu$, see \cite[Lemma 15]{DubedatSLEDuality}. We list some of them that will be helpful later. 
\begin{itemize}
\item If $\rho\ge\kappa/2-2$, the curve never hits the interval $[x,y)$. If $\rho+\nu\ge \kappa/2-2$, the curve never hits the interval $[y,\infty)$. 
\item If $\rho>-2$ and $\rho+\nu\in (\kappa/2-4, \kappa/2-2)$, the curve hits the interval $(y,\infty)$ at finite time. 
\item If $\rho>-2$ and $\rho+\nu\le \kappa/2-4$, the curve accumulates at the point $y$ at finite time.
\end{itemize}
By Girsanov's Theorem, the law of $\SLE_{\kappa}(\rho,\nu)$ process can be obtained by weighting the law of ordinary $\SLE_{\kappa}$, see \cite[Theorem 6]{SchrammWilsonSLECoordinatechanges}:
\begin{lemma}\label{lem::sle_rho_mart}
Fix $\kappa\in (0,8)$ and $0<x<y$. Let $\eta$ be an $\SLE_{\kappa}$ in $\HH$ from 0 to $\infty$ and $(g_t,t\ge 0)$ be the corresponding family of conformal maps. Define 
\begin{align*}
M_t(x,y)&=g_t'(x)^{\rho(\rho+4-\kappa)/(4\kappa)}(g_t(x)-W_t)^{\rho/\kappa}\\
&\times g_t'(y)^{\nu(\nu+4-\kappa)/(4\kappa)}(g_t(y)-W_t)^{\nu/\kappa}\times (g_t(y)-g_t(x))^{\rho\nu/(2\kappa)}.
\end{align*}
Then $M_t(x,y)$ is a local martingale for $\SLE_{\kappa}$ and the law of $\SLE_{\kappa}$ weighted by $M_t(x,y)$ is equal to the law of $\SLE_{\kappa}(\rho,\nu)$ with force points $(x,y)$ up to the swallowing time of $x$. 
\end{lemma}
\begin{lemma}\label{lem::sle_conditioned}
Fix $\kappa\in (4,8)$. For $y>1$, let $\PP_{y}$ be the law of $\SLE_{\kappa}$ conditioned not to hit the interval $(1,y)$. Then $\PP_{y}$ converges weakly to $\SLE_{\kappa}(\kappa-4)$ as $y\to\infty$. 
\end{lemma}
\begin{proof}
Let $\eta$ be an $\SLE_{\kappa}$ and denote by $\PP$ its law. 
By Lemma \ref{lem::sle_rho_mart}, we know that the process:
\[M_t:=(g_t(1)-W_t)^{(\kappa-4)/\kappa}\]
is a local martingale for $\eta$ and the law of $\eta$ weighted by $M$ becomes $\SLE_{\kappa}(\kappa-4)$ up to the first time that $\eta$ swallows $1$. Since $\SLE_{\kappa}(\kappa-4)$ does not hit the interval $(1,\infty)$, we know that $M_t$ is in fact a uniformly integrable martingale and the law of $\eta$ weighted by $M$ is $\SLE_{\kappa}(\kappa-4)$ for all time. For $N\ge 0$, define $\tau_N=\inf\{t: M_t=N\}$. By Optional Stopping Theorem, we have $\E[M_{\tau_N\wedge\tau_0}]=1$, thus the event $E_N=\{\tau_N<\tau_0\}$ has probability $1/N$. Therefore, weighting the law of $\eta$ by $M$ is equivalent to conditioning $\eta$ on $E_N$ up to $\tau_N\wedge\tau_0$. This is true for all $N$, thus weighting the law of $\eta$ by $M$ is equivalent to conditioning $\eta$ not to hit the interval $(1,\infty)$. 
\end{proof}
\subsection{Gaussian Free Field}
Suppose that $D\subsetneq \C$ is a proper domain with harmonically non-trivial boundary (i.e. a Brownian motion started at a point in $D$ hits $\partial D$ almost surely.) For $f,g\in L^2(D)$, we denote by $(f,g)$ the inner product of $L^2(D)$: $(f,g)=\int_D f(z)g(z)d^2z$,
where $d^2z$ is the Lebesgue area measure. Denote by $H_s(D)$ the space of real-valued smooth functions which are compactly supported in $D$. This space has a \textit{Dirichlet inner product} defined by
\[(f,g)_{\nabla}=\frac{1}{2\pi}\int_D\nabla f(z)\cdot\nabla g(z)d^2z.\]
Denote by $H(D)$ the Hilbert space completion of $H_s(D)$.

The \textit{zero-boundary $\GFF$} on $D$ is a random sum of the form
$h=\sum_{j=1}^{\infty}\alpha_jf_j$, 
where the $\alpha_j$ are i.i.d. one-dimensional standard Gaussians (with mean zero and variance 1) and the $f_j$ are an orthonormal basis for $H(D)$. This sum almost surely diverges within $H(D)$; however, it does converge almost surely in the space of distributions--- that is, the limit $\sum_j \alpha_j(f_j,p)$ almost surely exists for all $p\in H_s(D)$, and the limiting values, denoted by $(h,p)$, as a function of $p$ is almost surely a continuous functional on $H_s(D)$.
For any $f\in H_s(D)$, let $p=-\Delta f\in H_s(D)$, and define
$(h,f)_{\nabla}:=\frac{1}{2\pi}(h,p)$.
Then $(h,f)_{\nabla}$ is a mean-zero Gaussian with variance
\[\frac{1}{4\pi^2}\sum_j(f_j,p)^2=\sum_j(f_j,f)_{\nabla}^2=(f,f)_{\nabla}^2.\]
The zero-boundary $\GFF$ on $D$ is the only random distribution on $D$ with the property that, for each $f\in H_s(D)$, the element $(h,f)_{\nabla}$ is a mean-zero Gaussian with variance $(f,f)_{\nabla}$.
For any harmonic function $h_0$ on $D$, we use the phrase $\GFF$ with boundary data $h_0$ to indicate $h=\tilde{h}+h_0$ where $\tilde{h}$ is a zero-boundary $\GFF$. 

In this section, we will introduce level lines and flow lines of $\GFF$ and list their properties proved in \cite{SchrammSheffieldContinuumGFF, MillerSheffieldIG1, WangWuLevellinesGFFI}. 
Let $(K_t, t\ge 0)$ be an $\SLE_{\kappa}(\underline{\rho}^L;\underline{\rho}^R)$ process with force points $(\underline{x}^L;\underline{x}^R)$ where $W,V^{i,q}$ solves (\ref{eqn::sle_rho_sde}). Let $(g_t, t\ge 0)$ be the corresponding family of conformal maps and set $f_t=g_t-W_t$. 
Let $h_t^0$ be the harmonic function on $\HH$ with boundary values given by
\[\begin{cases}
-\lambda(1+\sum_0^j\rho^{i,L}),\quad &\text{if } x\in [V_t^{j+1,L}, V_t^{j,L}),\\
\lambda(1+\sum_0^j \rho^{i,R}), &\text{if } x\in [V_t^{j,R}, V_t^{j+1, R}),
\end{cases}\]
where $\lambda=\pi/\sqrt{\kappa}$ with the convention that $\rho^{0,L}=\rho^{0,R}=0, x^{0,L}=0_-, x^{l+1, L}=-\infty, x^{0,R}=0_+, x^{r+1, R}=\infty$. Define 
\[h_t(z)=h_t^0(f_t(z))-\chi\arg f_t'(z),\quad\text{where }\chi=2/\sqrt{\kappa}-\sqrt{\kappa}/2.\]
There exists a coupling $(h,K)$ where $\tilde{h}$ is a zero-boundary $\GFF$ on $\HH$ and $h=\tilde{h}+h_0$ such that the following is true. Suppose that $\tau$ is any $K$-stopping time before the continuation threshold. Then the conditional law of $h$ restricted to $\HH\setminus K_{\tau}$ given $K_{\tau}$ is the same as the law of $h_{\tau}+\tilde{h}\circ f_{\tau}$. In this coupling, the process $K$ is almost surely determined by $h$. When $\kappa\in (0,4)$, we refer to the $\SLE_{\kappa}(\underline{\rho}^L;\underline{\rho}^R)$ curve in this coupling as the \textit{flow line} of the field $h$; and for $\theta\in\R$, we use the phrase flow line of angle $\theta$ to indicate the flow line of $h+\theta\chi$. 
When $\kappa=4$, we refer to the $\SLE_{4}(\underline{\rho}^L;\underline{\rho}^R)$ in this coupling as the \textit{level line} of the field $h$; and for $u\in\R$, we use the phrase level line with height $u$ to indicate the level line of $h-u$. In this paper, we focus on $\kappa\in (0,4]$. We usually fix $\kappa\in (0,4)$ and set $\kappa'=16/\kappa$. For $\kappa'>4$, we refer to the $\SLE_{\kappa'}(\underline{\rho}^L;\underline{\rho}^R)$ curve coupled with $-h$ in the coupling as the \textit{counterflow line} of $h$. 

In the rest of this section, we fix the following constants:
\begin{equation}\label{eqn::gff_constants}
\kappa\in (0,4),\quad \kappa'=16/\kappa, \quad \lambda=\pi/\sqrt{\kappa},\quad \chi=2/\sqrt{\kappa}-\sqrt{\kappa}/2.
\end{equation}
The flow lines and counterflow lines of $\GFF$ interact in a nice way.
Suppose that $h$ is a $\GFF$ on $\HH$ with piecewise constant data. For $\theta\in\R$, let $\eta_{\theta}$ be the flow line of $h$ with angle $\theta$. Fix $\theta_1>\theta_2>\theta_3$ and suppose that $\eta_{\theta_1}$ and $\eta_{\theta_3}$ do not hit their continuation threshold. Then 
the flow line $\eta_{\theta_2}$ stays to the left of $\eta_{\theta_1}$ and stays to the right of $\eta_{\theta_3}$. Moreover, given $\eta_{\theta_1}$ and $\eta_{\theta_3}$, the conditional law of $\eta_{\theta_2}$ is $\SLE_{\kappa}(\rho^L;\rho^R)$ where 
\[\rho^L=-2+(\theta_1\chi-\theta_2\chi)/\lambda,\quad \rho^R=-2+(\theta_2\chi-\theta_3\chi)/\lambda.\]

Suppose that $h$ is a $\GFF$ on $\HH$ with piecewise constant boundary data. Let $\eta'$ be the counterflow line of $h$ from $\infty$ to $0$ and assume that the continuation threshold of $\eta'$ is not hit and $\eta'$ is nowhere boundary filling. Let $\eta_{+}$ be the flow line of $h$ with angle $\pi/2$ and $\eta_{-}$ be the flow line of $h$ with angle $-\pi/2$. Then $\eta_{+}$ is the left boundary of $\eta'$ and $\eta_{-}$ is the right boundary of $\eta'$. Combining these facts, we obtain the following decomposition of $\eta'$. 

\begin{lemma}\label{lem::counterflowline_decomposition}
Fix $\kappa\in (2,4)$ and $\kappa'=16/\kappa \in (4,8)$ and $\rho^L>-2, \rho^R>-2$. 
Let $\eta'$ be an $\SLE_{\kappa'}(\rho^L;\rho^R)$ in $\HH$ from $\infty$ to 0, and denote by $\eta_+$ its left boundary and $\eta_-$ its right boundary. Then we have the following.
\begin{itemize}
\item The law of $\eta_+$ is $\SLE_{\kappa}(\kappa-4+\kappa\rho^L/4;\kappa/2-2+\kappa\rho^R/4)$. 
\item Given $\eta_-$, the conditional law of $\eta_+$ is $\SLE_{\kappa}(\kappa-4+\kappa\rho^L/4; -\kappa/2)$.
\item Given $\eta_+$, the conditional law of $\eta'$ is $\SLE_{\kappa'}(\kappa'/2-4; \rho^R)$.
\item Given $\eta_+$ and $\eta_-$, the conditional law of $\eta'$ is $\SLE_{\kappa'}(\kappa'/2-4;\kappa'/2-4)$.
\end{itemize} 
\end{lemma}

\section{Hypergeometric SLE and Proof of Theorem \ref{thm::hyperSLE_reversibility}}
\label{sec::hyperSLE}
Fix $\kappa\in (0,8), \rho\in\R$ and two boundary points $0<x<y$. Recall that $F$ is the hypergeometric function defined in (\ref{eqn::hSLE_hyperF}). 
\textit{Hypergeometric SLE}, denoted by $\hSLE_{\kappa}(\rho)$, with marked points $(x,y)$ is the random Loewner chain driven by $W$ which is the solution to the following system of integrated SDEs:
\begin{align}\label{eqn::hyperSLE_sde}
W_t&=\sqrt{\kappa}B_t+\int_0^t\frac{(\rho+2)ds}{W_s-V^x_s}+\int_0^t\frac{-(\rho+2)ds}{W_s-V_s^y}-\int_0^t\kappa\frac{F'(Z_s)}{F(Z_s)}\left(\frac{1-Z_s}{V_s^y-W_s}\right)ds,\\
V_t^x&=x+\int_0^t\frac{2ds}{V_s^x-W_s},\quad V_t^y=y+\int_0^t\frac{2ds}{V_s^y-W_s}, \quad \text{where }Z_t=\frac{V^x_t-W_t}{V_t^y-W_t},\notag
\end{align}
where $B_t$ is one-dimensional Brownian motion. It is clear that the process is well-defined up to the swallowing time of $x$. Moreover, by Girsanov's Theorem, one can check that the law of $\hSLE_{\kappa}(\rho)$ with marked points $(x,y)$, up to the swallowing time of $x$, can be constructed by weighting the law of $\SLE_{\kappa}$ by the local martingale given in Lemma \ref{lem::hypersle_mart}. 

\begin{proposition}\label{prop::hyperSLE}
Fix $\kappa\in (0,8), \rho>(-4)\vee (\kappa/2-6)$ and $0<x<y$. The $\hSLE_{\kappa}(\rho)$ in $\HH$ from 0 to $\infty$ with marked points $(x,y)$ is well-defined for all time and it is almost surely generated by a continuous transient curve. Moreover, it never hits the interval $[x,y]$ when $\rho\ge \kappa/2-4$. 
\end{proposition}

Before proving Proposition \ref{prop::hyperSLE}, let us compare $\hSLE_{\kappa}(\rho)$ with $\SLE_{\kappa}(\rho+2,\kappa-6-\rho)$ process. 
By Girsanov's Theorem, one can check that 
the Radon-Nikodym derivative of the law of $\hSLE_{\kappa}(\rho)$ with marked points $(x,y)$ with respect to the law of $\SLE_{\kappa}(\rho+2,\kappa-6-\rho)$ with force points $(x,y)$ is given by 
\[R_t=\frac{F(Z_t)}{F(x/y)}\left(\frac{g_t(y)-W_t}{y}\right)^{4/\kappa-1},\quad \text{where }Z_t=\frac{g_t(x)-W_t}{g_t(y)-W_t}.\]
Note that $0\le Z_t\le 1$ for all $t$ and $F(z)$ is bounded for $z\in [0,1]$. 
Define, for $n\ge 1$, 
\[T_y^n=\inf\{t: g_t(y)-W_t\le 1/n \text{ or }g_t(y)-W_t\ge n\}.\]
Then we see that $R_{T_y^n}$ is bounded. Therefore, the law of $\hSLE_{\kappa}(\rho)$ is absolutely continuous with respect to the law of $\SLE_{\kappa}(\rho+2,\kappa-6-\rho)$ up to $T_y^n$. Since $\SLE_{\kappa}(\rho+2,\kappa-6-\rho)$ is generated by a continuous curve up to $T_y$ and it does not hit the interval $[x,y)$ when $\rho\ge\kappa/2-4$, we know that $\hSLE_{\kappa}(\rho)$ is generated by a continuous curve up to $T_y^n$ and it does not hit the interval $[x,y)$ up to $T_y^n$ when $\rho\ge \kappa/2-4$. Let $n\to\infty$, we see that $\hSLE_{\kappa}(\rho)$ is generated by continuous curve up to $T_y=\lim_n T_y^n$ and it does not hit the interval $[x,y)$ up to $T_y$ when $\rho\ge\kappa/2-4$. 

\begin{remark}\label{rem::hsle_degenerate}
From the above argument, we see that $\hSLE_{\kappa}(\rho)$ with marked points $(x,y)$ converges weakly to $\SLE_{\kappa}(\rho+2)$ with force point $x$ when $y\to\infty$ for $\kappa\in (0,8)$ and $\rho>(-4)\vee(\kappa/2-6)$. 
\end{remark}

Note that the absolute continuity of $\hSLE_{\kappa}(\rho)$ with respect to $\SLE_{\kappa}(\rho+2,\kappa-6-\rho)$ is not preserved as $n\to\infty$, since $R_t$ may be no longer bounded away from 0 or $\infty$ as $t\to T_y$. The following lemma discusses the behavior of $\hSLE_{\kappa}(\rho)$ as $t\to T_y$. 

\begin{lemma}\label{lem::hyperSLE_aroundTy}
When $\kappa\in (0,8)$ and $\rho>(-4)\vee(\kappa/2-6)$, 
the $\hSLE_{\kappa}(\rho)$ is well-defined and is generated by continuous curve up to and including the swallowing time of $y$, denoted by $T_y$. Moreover, the curve does not hit the interval $[x,y]$ if $\rho\ge \kappa/2-4$; and $T_y=\infty$ when $\kappa\le 4$; and the curve accumulates at a point in the interval $(y,\infty)$ as $t\to T_y<\infty$ when $\kappa\in (4,8)$. 
\end{lemma}

\begin{proof}
One can check that $\hSLE_{\kappa}(\rho)$ process $(K_t, t\ge 0)$ is scaling invariant: for any $\lambda>0$, the process $(\lambda K_{t/\lambda^2}, t\ge 0)$ has the same law as $\hSLE_{\kappa}(\rho)$ with marked points $(\lambda x, \lambda y)$. Thus, we may assume $y=1$ and $x\in (0,1)$, and denote $T_y$ by $T$. 
In this lemma, we discuss the behavior of $\hSLE_{\kappa}(\rho)$ as $t\to T$ and we will argue that the process does not accumulate at the point $1$. To this end, we perform a standard change of coordinate and parameterize the process according the capacity seen from the point $1$, see \cite[Theorem 3]{SchrammWilsonSLECoordinatechanges} or \cite[Section 4.3.3]{QianConformalRestrictionTrichordal}. 

Set $f(z)=z/(1-z)$. Clearly, $f$ is the conformal M\"{o}bius transform of $\HH$ sending the points $(0,1,\infty)$ to $(0,\infty, -1)$. Consider the image of $(K_t, 0\le t\le T)$ under $f$: $(\tilde{K}_s, 0\le s\le \tilde{S})$ where we parameterize this curve by its capacity $s(t)$ seen from $\infty$. Let $(\tilde{g}_s)$ be the corresponding family of conformal maps and $(\tilde{W}_s)$ be the driving function. Let $f_t$ be the M\"{o}bius transform of $\HH$ such that $\tilde{g}_s\circ f=f_t\circ g_t$ where $s=s(t)$. By expanding $\tilde{g}_s=f_t\circ g_t\circ f^{-1}$ around $\infty$ and comparing the coefficients in both sides, we have
\[f_t(z)=-1-\frac{g_t''(1)}{2g_t'(1)}+\frac{g_t'(1)}{g_t(1)-z}.\]
Thus, with $s=s(t)$, 
\[ \tilde{W}_s=f_t(W_t)=-1-\frac{g_t''(1)}{2g_t'(1)}+\frac{g_t'(1)}{g_t(1)-W_t},\quad d\tilde{W}_s=\frac{(\kappa-6)g_t'(1)dt}{(g_t(1)-W_t)^3}+\frac{g_t'(1)dW_t}{(g_t(1)-W_t)^2}.\]
Define 
\[\tilde{V}^x_s=f_t(V_t^x), \quad \tilde{V}^{\infty}_s=f_t(\infty),\quad \tilde{Z}_s=\frac{\tilde{V}^x_s-\tilde{W}_s}{\tilde{V}^x_s-\tilde{V}^{\infty}_s}=Z_t.\]
Plugging in the time change
\[\dot{s}(t)=f_t'(W_t)^2=\frac{g_t'(1)^2}{(g_t(1)-W_t)^4},\]
we obtain
\[d\tilde{W}_s=\sqrt{\kappa}d\tilde{B}_s+\frac{(\rho+2)ds}{\tilde{W}_s-\tilde{V}^x_s}+\frac{(\kappa-6)ds}{\tilde{W}_s-\tilde{V}^{\infty}_s}-\kappa\frac{F'(\tilde{Z}_s)}{F(\tilde{Z}_s)}\frac{ds}{\tilde{V}^x_s-\tilde{V}^{\infty}_s},\]
where $\tilde{B}_s$ is one-dimensional Brownian motion. By Girsanov's Theorem, the Radon-Nikodym derivative of the law of $\tilde{K}$ with respect to the law of $\SLE_{\kappa}(\kappa-6;\rho+2)$ with force points $(-1; \tilde{x}:=x/(1-x))$ is given by 
\[R_s=\frac{F(Z_s)}{F(x)}\left(\frac{g_s(\tilde{x})-g_s(-1)}{(1-x)^{-1}}\right)^{-(\rho+2)/\kappa},\quad \text{where }Z_s=\frac{g_s(\tilde{x})-W_s}{g_s(\tilde{x})-g_s(-1)}.\]
Note that $0\le Z_s\le 1$ and $F(z)$ is bounded for $z\in [0,1]$; and that the process $g_s(\tilde{x})-g_s(-1)$ is increasing, thus $g_s(\tilde{x})-g_s(-1)\ge 1/(1-x)$. Let $S$ be the swallowing time of $-1$. Define, for $n\ge 1$,
\[S^n=\inf\{t: K_t\text{ exits }B(0,n)\}.\]
Then $R_s$ is bounded up to $S\wedge S^n$, and thus the process $\tilde{K}$ is absolutely continuous with respect to $\SLE_{\kappa}(\kappa-6;\rho+2)$ up to $S\wedge S^n$. We list some properties of $\SLE_{\kappa}(\kappa-6;\rho+2)$ with force points $(-1; \tilde{x}=x/(1-x))$ here: it is generated by continuous curve up to and including the continuation threshold; 
the curve does not hit the interval $[\tilde{x},\infty)$ when $\rho+2\ge\kappa/2-2$. 
When $\kappa\in (0,4]$, the curve almost surely accumulates at the point $-1$, since $\kappa-6\le \kappa/2-4$; when $\kappa\in (4,8)$, the curve hits the interval $(-\infty, -1)$ at finite time almost surely, since $\kappa-6\in (-2,\kappa/2-2)$.
Therefore, the process $\tilde{K}$ is generated by continuous curve up to and including $\tilde{S}$. This implies that our original $\hSLE_{\kappa}(\rho)$ process $(K_t, t\ge 0)$ is generated by continuous curve up to and including $T$; moreover, the curve accumulates at a point in $(y,\infty)\cup\{\infty\}$ as $t\to T$, and it does not hit $[x,y]$ when $\rho\ge\kappa/2-4$. 
\end{proof}

\begin{proof}[Proof of Proposition \ref{prop::hyperSLE}]
In Lemma \ref{lem::hyperSLE_aroundTy}, we have shown that $\hSLE_{\kappa}(\rho)$ is well-defined and is generated by continuous curve up to and including $T_y$. In particular, when $\kappa\le 4$, since $T_y=\infty$, we obtain the conclusion for this case. It remains to prove the conclusion for $\kappa\in (4,8)$. In this case, as $t\to T_y$, we have 
$V^y_t-W_t\to 0$ and $Z_t\to 1$. Note that $F(z)$ remains bounded as $z\to 1$; and that $F'(z)(1-z)\to 0$ as $z\to 1$, since (see \cite[Equations (15.2.1),(15.3.3)]{AbramowitzHandbook}), as $z\to 1$,
\[F'(z)=\left(\frac{\rho+2}{\rho+4}\right)\left(1-\frac{4}{\kappa}\right)(1-z)^{8/\kappa-2}\hF\left(\frac{4}{\kappa}, \frac{12+2\rho}{\kappa}-1, \frac{8+2\rho}{\kappa}+1; z\right)\approx (1-z)^{8/\kappa-2}.\]
Combining these, we know that the SDE (\ref{eqn::hyperSLE_sde}) degenerates to $W_t=\sqrt{\kappa}B_t$ for $t\ge T_y$. Therefore, the process is the same as standard $\SLE_{\kappa}$ for $t\ge T_y$, and hence is generated by continuous transient curve. 
\end{proof}

\begin{proposition}\label{prop::hypersle_mart}
Fix $\kappa\in (0,8), \rho\ge \kappa/2-4$ and $0<x<y$. The local martingale defined in Lemma \ref{lem::hypersle_mart} is a uniformly integrable martingale for $\SLE_{\kappa}$; and the law of $\SLE_{\kappa}$ weighted by this martingale is the same as $\hSLE_{\kappa}(\rho)$ with marked points $(x,y)$.
\end{proposition}
\begin{proof}[Proof of Proposition \ref{prop::hypersle_mart} and Theorem \ref{thm::hyperSLE_reversibility}]
In Lemma \ref{lem::hypersle_mart}, we have shown that $M_t$ is a local martingale up to the swallowing time of $x$. Note that $J_t$ is decreasing in $t$, thus $J_t\le J_0$. Therefore $M_t$ is bounded as long as $J_t$ and $Z_t$ are bounded from below. 
Define, for $n\ge 1$,
\[T^n=\inf\{t: J_t\le 1/n\text{ or }Z_t\le 1/n\}.\]
Then $M_{t\wedge T^n}$ is a bounded martingale; moreover, the law of $\SLE_{\kappa}$ weighted by $M_{t}$ is the law of $\hSLE_{\kappa}(\rho)$ up to $T^n$. By Proposition \ref{prop::hyperSLE}, we know that $\hSLE_{\kappa}(\rho)$ is generated by a continuous transient curve and the curve never hits the interval $[x, y]$. Therefore, $M_t$ is actually a uniformly integrable martingale for $\SLE_{\kappa}$. This completes the proof of Proposition \ref{prop::hypersle_mart}. 

We have shown that $\hSLE_{\kappa}(\rho)$ is generated by continuous curve in Proposition \ref{prop::hyperSLE}, to show Theorem 
\ref{thm::hyperSLE_reversibility}, it remains to show the reversibility. To this end, we will derive the explicit formula for $M_{\infty}$. Given a deterministic continuous curve $\eta$ in $\overline{\HH}$ from 0 to $\infty$ with continuous driving function that does not hit the interval $[x,y]$, denote by $D(x,y)$ the connected component of $\HH\setminus \eta$ with $[x,y]$ on the boundary. We know that 
\[\lim_{t\to \infty}Z_t=1,\quad \lim_{t\to\infty}J_t=J_{\infty}:=\frac{g'(x)g'(y)}{(g(y)-g(x))^2},\]
where $g$ is any conformal map from $D(x,y)$ onto $\HH$. One can check that the quantity $J_{\infty}$ only depends on the region $D(x,y)$ and does not depend on the choice of conformal map $g$. In fact, the quantity $J_{\infty}$ is the so-called Poisson kernel of the region $D(x,y)$. Thus we have almost surely $M_{\infty}=\lim_{t\to \infty}M_t=J_{\infty}^b$. Moreover, the Radon-Nikodym derivative of the law of $\hSLE_{\kappa}(\rho)$ with marked points $(x,y)$ with respect to the law of $\SLE_{\kappa}$ is given by $M_{\infty}/M_0$. Combining the reversibility of standard $\SLE_{\kappa}$ and the conformal invariance of the quantity $M_{\infty}/M_0$, we have the reversibility of $\hSLE_{\kappa}(\rho)$.  
\end{proof}

To end this section, we will prove a generalization of Lemma \ref{lem::sle_conditioned}. 

\begin{proposition}
Fix $\kappa\in (0,8), \rho\in ((-4)\vee(\kappa/2-6), \kappa/2-4)$, and $0<x<y$. Let $\eta$ be an $\hSLE_{\kappa}(\rho)$ in $\HH$ from 0 to $\infty$ with marked points $(x,y)$. The law of $\eta$ conditioned to avoid the interval $[x,y]$ is the same as $\hSLE_{\kappa}(\kappa-8-\rho)$.  In particular, the law of $\SLE_{\kappa}$ with $\kappa\in (4,8)$ conditioned to avoid the interval $[x,y]$ is the same as $\hSLE_{\kappa}(\kappa-6)$ with marked points $(x,y)$. 
\end{proposition}

\begin{proof}
Let $\hat{\eta}$ be an $\hSLE_{\kappa}(\kappa-8-\rho)$ in $\HH$ from 0 to $\infty$ with marked points $(x,y)$. By SDE (\ref{eqn::hyperSLE_sde}), we see that the law of $\hat{\eta}$ is the same as the law of $\eta$ weighted by the following martingale up to the swallowing time of $x$, denoted by $T_x$:
\[M_t=Z_t^{(\kappa-8-2\rho)/(4\kappa)}\hat{F}(Z_t)/F(Z_t)\one_{\{t<T_x\}}, \quad \text{where }Z_t=\frac{g_t(x)-W_t}{g_t(y)-W_t},\]
and 
\[F(z)=\hF\left(\frac{2\rho+4}{\kappa}, 1-\frac{4}{\kappa}, \frac{2\rho+8}{\kappa};z\right),\quad \hat{F}(z)=\hF\left(\frac{2(\kappa-6-\rho)}{\kappa}, 1-\frac{4}{\kappa}, \frac{2(\kappa-4-\rho)}{\kappa}; z\right).\]
Note that $\hat{F}$ and $F$ are positive and bounded for $z\in [0,1]$.

Since $\kappa-8-\rho\ge\kappa/2-4$, we know that the law of $\eta$ weighted by $M$ does not hit the interval $[x,y]$ up to $T_x$. In other words, under the weighted law, we have $T_x$ coincides with the swallowing time of $y$, denoted by $T_y$. Hence, the law of $\hat{\eta}$ is the same as the law of $\eta$ weighted by $M$ up to $T_y$. As $t\to T_y$, we have $Z_t\to 1$. Therefore, $M$ is a uniformly integrable martingale for $\eta$, and the law of $\hat{\eta}$ is the same as the law of $\eta$ weighted by $\one_{\{\eta\cap [x,y]=\emptyset\}}$.
This completes the proof. 
\end{proof}

\section{Proof of Theorem \ref{thm::slepair_lessthan4}}
\label{sec::slepair}
\subsection{Proof of Theorem \ref{thm::slepair_lessthan4}---Uniqueness}
\begin{proposition}\label{prop::slepair_uniqueness_general}
Fix a quad $(\Omega; x^L, x^R, y^R, y^L)$ and $\kappa\in (0,8), \rho^L>-2,\rho^R>-2,\nu\ge\kappa/2-2$. 
There exists at most one probability measure on pairs of curves $(\eta^L;\eta^R)$ in $X_0(\Omega; x^L, x^R, y^R, y^L)$ with the following property: the conditional law of $\eta^R$ given $\eta^L$ is $\SLE_{\kappa}(\nu;\rho^R)$ in the connected component of $\Omega\setminus\eta^L$ with $(x^Ry^R)$ on the boundary, and the conditional law of $\eta^L$ given $\eta^R$ is $\SLE_{\kappa}(\rho^L;\nu)$ in the connected component of $\Omega\setminus \eta^R$ with $(y^Lx^L)$ on the boundary. 
\end{proposition}
This proposition was proved for $\kappa\in (0,4], \nu=0$ in \cite[Theorem 4.1]{MillerSheffieldIG2} and the same proof works as long as the two curves do not hit each other. To be self-contained, we will give a brief proof here and point out why this proof only works when the two curves do not hit.

Before proving this proposition, let us first explain that the conclusion for $\kappa\in (4,8)$ follows easily from the conclusion for $\kappa\in (0,4]$. Assume the conclusion in Proposition \ref{prop::slepair_uniqueness_general} is true for $\kappa\in (0,4]$. Fix $\kappa'\in (4,8)$ and set $\kappa=16/\kappa'\in (2,4)$. Suppose that $(\eta^L;\eta^R)$ is a pair in $X_0(\Omega; x^L, x^R, y^L, y^R)$ such that the conditional law of $\eta^R$ given $\eta^L$ is $\SLE_{\kappa'}(\nu;\rho^R)$ and the conditional law of $\eta^L$ given $\eta^R$ is $\SLE_{\kappa'}(\rho^L;\nu)$ where $\rho^L, \rho^R>-2$ and $\nu\ge\kappa'/2-2$. Let $\eta^L_-$ be the right boundary of $\eta^L$ and let $\eta^R_+$ be the left boundary of $\eta^R$. Since the conditional law of $\eta^R$ given $\eta^L_-$ is $\SLE_{\kappa'}(\nu;\rho^R)$, we know that the conditional law of  $\eta^R_+$ given $\eta^L_-$ is (see Lemma \ref{lem::counterflowline_decomposition})
\[\SLE_{\kappa}(\kappa-4+\kappa\nu/4; \kappa/2-2+\kappa\rho^R/4).\]
Similarly, the conditional law of $\eta^L_-$ given $\eta^R_+$ is 
\[\SLE_{\kappa}(\kappa/2-2+\kappa\rho^L/4; \kappa-4+\kappa\nu/4).\]
Note that 
\[\kappa/2-2+\kappa\rho^L/4>-2,\quad \kappa/2-2+\kappa\rho^R/4>-2,\quad \kappa-4+\kappa\nu/4\ge \kappa/2-2.\]
By the conclusion in Proposition \ref{prop::slepair_uniqueness_general} for $\kappa\le 4$, we know that there is at most one probability measure on the pair $(\eta^L_-; \eta^R_+)$. Given the pair $(\eta^L_-; \eta^R_+)$, there is only one way to reconstruct the pair $(\eta^L;\eta^R)$, since the conditional law of $\eta^L$ given $(\eta^L_-;\eta^R_+)$ is $\SLE_{\kappa'}(\rho^L; \kappa'/2-4)$; and the conditional law of $\eta^R$ given $(\eta^L_-; \eta^R_+)$ is $\SLE_{\kappa'}(\kappa'/2-4; \rho^R)$. This implies that the conclusion in Proposition \ref{prop::slepair_uniqueness_general} holds for $\kappa'\in (4,8)$. 
\begin{proof}[Proof of Proposition \ref{prop::slepair_uniqueness_general}]
From the above argument, we only need to prove the conclusion for $\kappa\in (0,4]$. 
Define a Markov chain which transitions from a configuration $(\eta^L;\eta^R)$ to another $(\tilde{\eta}^L;\tilde{\eta}^R)$ in the following way: 
Given a configuration $(\eta^L;\eta^R)$ of non-intersecting curves in $X_0(\Omega; x^L, x^R, y^R, y^L)$, we pick $q\in\{L, R\}$ uniformly and resample $\eta^q$ according to the conditional law of $\eta^q$ given the other one. We will argue that this chain has at most one stationary measure. 
Suppose that $\mu$ is any stationary measure for this chain. Fix $\eps>0$ small, and let $\mu_{\eps}$ be the measure $\mu$ conditioned on $X_0^{\eps}(\Omega; x^L, x^R, y^R, y^L)$. Then $\mu_{\eps}$ is stationary for the $\eps$-Markov chain: the chain is defined the same as before except in each step we resample the path conditioned on $X_0^{\eps}(\Omega; x^L, x^R, y^R, y^L)$. Let $\Sigma^{\eps}$ be the set of all such stationary measures. Clearly, $\Sigma^{\eps}$ is convex. 
\smallbreak
\textit{First}, we argue that $\Sigma^{\eps}$ is compact. Suppose that $\nu_n$ is a sequence in $\Sigma^{\eps}$ converging weakly to $\nu$, we need to show that $\nu$ is also stationary for the $\eps$-Markov chain. Suppose that $(\eta^L_n;\eta^R_n)$ has law $\nu_n$ and $(\eta^L;\eta^R)$ has law $\nu$. By Skorohod Representation Theorem, we could couple all $(\eta^L_n;\eta^R_n)$ and $(\eta^L;\eta^R)$ in a common space so that $\eta^L_n\to \eta^L$ and $\eta^R_n\to\eta^R$ almost surely. Let $D_n^R$ be the connected component of $\Omega\setminus\eta^L_n$ with $(x^Ry^R)$ on the boundary; and $D_n^L$ be the connected component of $\Omega\setminus\eta^R_n$ with $(y^Lx^L)$ on the boundary. Define $D^L, D^R$ for $(\eta^L;\eta^R)$ similarly. For $\delta>0$ small, let $U^R_{\delta}$ be the open set of points in $D^R$ that has distance at least $\delta$ to $\eta^L$ and define $U^L_{\delta}$ in a similar way. For $q\in\{L, R\}$, let $\LF^q=\sigma(\eta^q_n,\eta^q, n\ge 1)$. By the convergence of $\eta^q_n\to\eta^q$, we know that $U^q_n\subset D^q$ for $q\in\{L, R\}$ for $n$ large enough. 

For each $n$, let $h^L_n$ be the $\GFF$ in $D^L_n$ with the boundary value so that its flow line from $x^L$ to $y^L$ is $\SLE_{\kappa}(\rho^L;\nu)$. Define $h^R_n, h^L, h^R$ analogously. We assume that $h^L$ and $h^L_n$ for all $n$ are coupled together so that, given $\LF^R$, they are conditionally independent. The same is true for $h^R_n, h^R$.
Given $\LF^R$, the total variation distance between the law of $h^L_n$ restricted to $U^L_{\delta}$ and the law of $h^L$ restricted to $U^L_{\delta}$ tends to 0; and similar conclusion also holds for $\LF^L, h^R_n, h^R$, see \cite[Equation (4.1)]{MillerSheffieldIG2}:
\[\lim_{n\to\infty}||\LL\left[h^L_n|_{U^L_{\delta}}\cond \LF^R\right]-\LL\left[h^L|_{U^L_{\delta}}\cond \LF^R\right]||_{TV}=0,\quad \lim_{n\to\infty}||\LL\left[h^R_n|_{U^R_{\delta}}\cond \LF^L\right]-\LL\left[h^R|_{U^R_{\delta}}\cond \LF^L\right]||_{TV}=0.\]
We will deduce that, given $\LF^L$, the flow line from $x^R$ to $y^R$ generated by $h_n^R$ converges to the one generated by $h^R$. Fix $\eps'>0$, since $\nu\ge\kappa/2-2$, there exists $\delta>0$ such that, given $\LF^L$, the flow line $\eta^R$ generated by $h^R$ is contained in $U^R_{\delta}$ with probability at least $1-\eps'$ (This is the part of proof that requires the two curves to be non-intersecting).
By the total variation convergence, we could choose $n_0$ such that, for $n\ge n_0$,  
\[||\LL\left[h^R_n|_{U^R_{\delta}}\cond \LF^L\right]-\LL\left[h^R|_{U^R_{\delta}}\cond \LF^L\right]||_{TV}\le \eps'.\]
Since the flow lines are deterministic function of the $\GFF$, the total variation distance between the two flow lines given $\LF^L$ is at most $2\eps'$. This implies that, given $\LF^L$, the total variation distance between the flow line generated by $h_n^R$ and the one generated by $h^R$ goes to zero. Similar result also holds for $\LF^R, h^L_n, h^L$. Since total variation convergence implies weak convergence, we have that the transition kernel for the $\eps$-Markov chain is continuous. Therefore, the measure $\nu$ is stationary. This completes the proof that $\Sigma^{\eps}$ is compact.   
\smallbreak
\textit{Second}, we show that $\Sigma^{\eps}$ is characterized by its extremals. Since $\Sigma^{\eps}$ is compact and the space of probability measures on $X^{\eps}_0(\Omega; x^L, y^L, y^R, x^R)$ is complete and separable, Choquet's Theorem \cite[Section 3]{Choquet} implies that $\mu_{\eps}$ can be uniquely expressed as a superposition of extremals in $\Sigma^{\eps}$. To show that $\Sigma^{\eps}$ consists of at most one element, it suffices to show that there is only one such extremal in $\Sigma^{\eps}$. Suppose that $\nu,\tilde{\nu}$ are two distinct extremal elements in $\Sigma^{\eps}$. Lebesgue's Decomposition Theorem tells that there is a unique decomposition $\nu=\nu_0+\nu_1$ where $\nu_0$ is absolutely continuous with respect to $\tilde{\nu}$ and $\nu_1$ is singular to $\tilde{\nu}$. If both $\nu_0,\nu_1$ are non-zero, then they can be normalized to probability measures in $\Sigma^{\eps}$, this contradicts that $\nu$ is extremal. Therefore, either $\nu$ is absolutely continuous with respect to $\tilde{\nu}$ or $\nu$ is singular to $\tilde{\nu}$. 
We could argue that $\nu$ can not be absolutely continuous with respect to $\tilde{\nu}$. This is proved in \cite[Proof of Theorem 4.1]{MillerSheffieldIG2}. 
\smallbreak
\textit{Finally,} we only need to show that $\nu$ and $\tilde{\nu}$ can not be singular. Suppose that we have two initial configurations $(\eta_0^L;\eta_0^R)\sim \nu$ and $(\tilde{\eta}_0^L;\tilde{\eta}_0^R)\sim \tilde{\nu}$ sampled independently. First, we set $\eta^L_1=\eta^L_0$ and $\tilde{\eta}^L_1=\tilde{\eta}^L_0$, and then, given $\eta^L_1$ and $\tilde{\eta}^L_1$, we sample $\eta^R_1$ and $\tilde{\eta}^R_1$ according to the conditional law and couple them to maximize the probability for them to be equal. The fact that this probability is positive is guaranteed by Lemma \ref{lem::sle_close_positive}. Next, we set $\eta^R_2=\eta^R_1$ and $\tilde{\eta}^R_2=\tilde{\eta}^R_1$, and then, given $\eta^R_2, \tilde{\eta}^R_2$, we sample $\eta^L_2, \tilde{\eta}^L_2$ according to the conditional law and couple them to maximize the probability for them to be equal. Lemma \ref{lem::sle_close_positive} guarantees that the probability for $(\eta^L_2;\eta^R_2)=(\tilde{\eta}^L_2;\tilde{\eta}^R_2)$ is positive, which implies that $\nu$ and $\tilde{\nu}$ can not be singular. This completes the proof that $\Sigma^{\eps}$ contains at most one element. Since this is true for any $\eps>0$, we know that there is at most one stationary measure for the original Markov chain.  
\end{proof}
\begin{lemma}\label{lem::sle_close_positive}
Let $(D; x,y)$ and $(\tilde{D}; x,y)$ be two Dobrushin domains such that $\partial \tilde{D}$ agrees with $\partial D$ in neighborhoods of the arc $(xy)$. Fix $\kappa>0, \rho^L>(\kappa/2-4)\vee -2$ and $\rho^R>-2$. Let $\eta$ (resp. $\tilde{\eta}$) be an $\SLE_{\kappa}(\rho^L;\rho^R)$ in $D$ (resp. $\tilde{D}$) from $x$ to $y$ with force points $(x_-;x_+)$. Then there exists a coupling $(\eta, \tilde{\eta})$ such that the probability for them to be equal is positive. 
\end{lemma}
\begin{proof}
Proof of \cite[Lemma 4.2]{MillerSheffieldIG2} and the discussion after the proof there. 
\end{proof}
\subsection{Proof of Theorem \ref{thm::slepair_lessthan4}---Existence and Identification}
Suppose $K$ is an $\HH$-hull such that $\dist(0,K)>0$ and $\R\cap K\subset (0,\infty)$. Let $H:=\HH\setminus K$. Let $\phi$ be the conformal map from $H$ onto $\HH$ such that $\phi(0)=0$ and $\lim_{z\to\infty}\phi(z)/z=1$.
We wish to compare the law of $\SLE_{\kappa}(\rho)$ with force point $v\le 0$ in $\HH$ and the law of $\SLE_{\kappa}(\rho)$ with force point $v\le 0$ in $H$. 
Suppose $\eta$ is an $\SLE_{\kappa}(\rho)$ with force point $v\le 0$ in $\HH$ from 0 to $\infty$ and $(g_t,t\ge 0)$ is the corresponding family of conformal maps, $W_t, t\ge 0$ is the driving function and $(V_t, t\ge 0)$ is the evolution of the force point $v\le 0$. Let $T$ be the first time that $\eta$ hits $K$.  We will study the law of $\tilde{\eta}(t)=\phi(\eta(t))$ for $t<T$. Define $\tilde{g}_t$ to be the conformal map from $\HH\setminus\tilde{\eta}[0,t]$ onto $\HH$ normalized at $\infty$ and let $h_t$ be the conformal map from $\HH\setminus g_t(K)$ onto $\HH$ such that $h_t\circ g_t=\tilde{g}_t\circ\phi$, see \cite[Section 5]{LawlerSchrammWernerConformalRestriction}. Note that $\tilde{W}_t=h_t(W_t)$ is the driving function for $\tilde{\eta}$ and 
\[d\tilde{W}_t=h_t'(W_t)dW_t+(\kappa/2-3)h_t''(W_t)dt.\]
\textit{Schwarzian derivative} for a conformal map $f$ is defined to be 
\[Sf(z)=\frac{f'''(z)}{f'(z)}-\frac{3f''(z)^2}{2f'(z)^2}.\]
\begin{lemma}\label{lem::sle_domain_mart}
Fix $\kappa\in (0,4], \rho>-2$. Let $\eta$ be an $\SLE_{\kappa}(\rho)$ in $\HH$ with force point $v\le 0$. 
The following process is a uniformly integrable martingale:
\[M_t:=\one_{\{t<T\}}h'_t(W_t)^{b_1}h_t'(V_t)^{b_2}\left(\frac{h_t(W_t)-h_t(V_t)}{W_t-V_t}\right)^{b_3}\exp\left(-c\int_0^t\frac{Sh_s(W_s)}{6}ds\right),\]
where 
\[b_1=\frac{6-\kappa}{2\kappa},\quad b_2=\frac{\rho(\rho+4-\kappa)}{4\kappa},\quad b_3=\frac{\rho}{\kappa}, \quad c=\frac{(3\kappa-8)(6-\kappa)}{2\kappa}.\]
Moreover, the Radon-Nikodym derivative between the law of $\SLE_{\kappa}(\rho)$ with force point $0_-$ in $H$ with respect to $\SLE_{\kappa}(\rho)$ with force point $0_-$ in $\HH$ is given by
\[M_{\infty}/M_0=\one_{\{\eta\cap K=\emptyset\}}\phi'(0)^{-b}\exp\left(-c\int_0^{\infty}\frac{Sh_s(W_s)}{6}ds\right),\]
where 
\[b=b_1+b_2+b_3=\frac{(\rho+2)(\rho+6-\kappa)}{4\kappa}.\] 
\end{lemma}

\begin{proof}
One can check by It\^{o}'s Formula that $M$ is a local martingale (see \cite[Lemma 1]{DubedatSLEMart}). The rest of the lemma was proved in \cite[Section 3]{WernerWuCLEtoSLE}. 
\end{proof}

The \textit{Brownian loop measure} is the measure on unrooted Brownian loops. Since we do not need to do calculation with it, we omit the introduction to Brownian loop measure here and refer \cite[Sections 3,4]{LawlerWernerBrownianLoopsoup} for a clear definition. Given a non-empty simply connected domain $\Omega\subsetneq\C$ and two disjoint subsets $V_1, V_2$, denote by $m(\Omega; V_1, V_2)$ the Brownian loop measure of loops in $\Omega$ that intersect both $V_1$ and $V_2$. This quantity is conformal invariant: $m(f(\Omega); f(V_1), f(V_2))=m(\Omega; V_1, V_2)$ for any conformal transformation $f$ on $\Omega$. When both of $V_1, V_2$ are closed, one of them is compact and $\dist(V_1, V_2)>0$, we have $0<m(\Omega; V_1, V_2)<\infty$. It is proved in \cite[Equation (22)]{LawlerWernerBrownianLoopsoup} that$-(1/6)\int_0^{t}Sh_s(W_s)ds=m(\HH; K, \eta[0,t])$.
Thus we have the Radon-Nikodym derivative between the law of $\SLE_{\kappa}(\rho)$ with force point $0_-$ in $H$ with respect to $\SLE_{\kappa}(\rho)$ with force point $0_-$ in $\HH$ is given by
\begin{equation}\label{eqn::sle_domain_brownianloopmass}
M_{\infty}/M_0=\one_{\{\eta\cap K=\emptyset\}}\phi'(0)^{-b}\exp(c m(\HH; K, \eta)).
\end{equation}

\begin{proof}[Proof of Theorem \ref{thm::slepair_lessthan4}, Existence and Identification] 
\textit{First}, we will construct a probability measure on $(\eta^L;\eta^R)\in X_0(\Omega; x^L, x^R, y^R, y^L)$. 
Fix $\kappa\in (0,4], \rho^L>-2, \rho^R>-2$ and $0<x<y$. By conformal invariance, it is sufficient to give the construction for the quad $(\HH; 0,x,y,\infty)$. 
Denote by $\PP_L$ the law of $\SLE_{\kappa}(\rho^L)$ in $\HH$ from 0 to $\infty$ with force point $0_-$ and denote by $\PP_R$ the law of $\SLE_{\kappa}(\rho^R)$ in $\HH$ from $x$ to $y$ with force point $x_+$. Define the measure $\mu$ on $X_0(\HH; 0,x,y,\infty)$ by 
\[\mu[d\eta^L, d\eta^R]=\one_{\{\eta^L\cap\eta^R=\emptyset\}}\exp\left(c m(\HH; \eta^L, \eta^R)\right)\PP_L\left[d\eta^L\right]\otimes\PP_R\left[d\eta^R\right].\] 
We argue that the total mass of $\mu$, denoted by $|\mu|$, is finite. Given $\eta^L\in X_0(\HH; 0,\infty)$, let $g$ be any conformal map from the connected component of $\HH\setminus \eta^L$ with $(xy)$ on the boundary onto $\HH$, then we have
\begin{align*}
|\mu|&=\E_L\otimes\E_R\left[\one_{\{\eta^L\cap\eta^R=\emptyset\}}\exp\left(c m(\HH; \eta^L, \eta^R)\right)\right]\\
&=\E_L\left[\left(\frac{g'(x)g'(y)}{(g(x)-g(y))^2}\right)^b\right]\tag{By Lemma \ref{lem::sle_domain_mart} and (\ref{eqn::sle_domain_brownianloopmass})}\\
&\le (y-x)^{-2b}. \tag{where $b=(\rho^R+2)(\rho^R+6-\kappa)/(4\kappa)$.}
\end{align*}
This implies that $|\mu|$ is positive finite. We define the probability measure $\PP$ to be $\mu/|\mu|$. 
\smallbreak
\textit{Second}, we show that, under $\PP$, the conditional law of $\eta^R$ given $\eta^L$ is $\SLE_{\kappa}(\rho^R)$. By the symmetry in the definition of $\PP$, we know that the conditional law of $\eta^L$ given $\eta^R$ is $\SLE_{\kappa}(\rho^L)$ and hence $\PP$ satisfies the property in ``Uniqueness". Given $\eta^L$, denote by $H$ the connected component of $\HH\setminus \eta^L$ with $(xy)$ on the boundary and let $g$ be any conformal map from $H$ onto $\HH$. Denote by $\PP_R$ the law of $\SLE_{\kappa}(\rho^R)$ in $\HH$ from $x$ to $y$ and by $\tilde{\PP}_R$ the law of $\SLE_{\kappa}(\rho^R)$ in $H$ from $x$ to $y$. By Lemma \ref{lem::sle_domain_mart},
for any bounded continuous function $\LF$ on continuous curves, we have 
\begin{align*}
\E\left[\LF(\eta^R)\cond \eta^L\right]
&=|\mu|^{-1}\E_R\left[\one_{\{\eta^L\cap\eta^R=\emptyset\}}\exp\left(c m(\HH; \eta^L, \eta^R)\right)\LF(\eta^R)\right]\\
&=|\mu|^{-1}\left(\frac{g'(x)g'(y)}{(g(x)-g(y))^2}\right)^b\tilde{\E}_R\left[\LF(\eta^R)\right].
\end{align*}
This implies that the conditional law of $\eta^R$ given $\eta^L$ is $\SLE_{\kappa}(\rho^R)$ in $H$. 
\smallbreak
\textit{Finally}, we show that, under $\PP$ and fixing $\rho^L=0$, the marginal law of $\eta^L$ is $\hSLE_{\kappa}(\rho^R)$. In fact, the above equation implies that the law of $\eta^L$ is the law of $\SLE_{\kappa}$ in $\HH$ from 0 to $\infty$ weighted by \[\left(\frac{g'(x)g'(y)}{(g(x)-g(y))^2}\right)^b.\]
Therefore, by Proposition \ref{prop::hypersle_mart} and the argument in its proof, we see that the law of $\eta^L$ coincides with $\hSLE_{\kappa}(\rho^R)$ as desired.
\end{proof}

\begin{figure}[ht!]
\begin{subfigure}[b]{0.19\textwidth}
\begin{center}
\includegraphics[width=\textwidth]{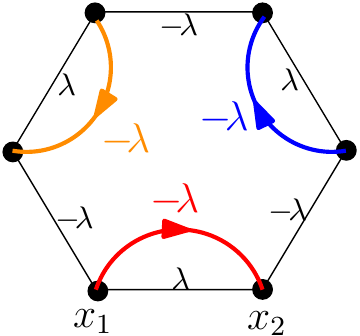}
\end{center}
\caption{\{12,34,56\}}
\end{subfigure}
\begin{subfigure}[b]{0.19\textwidth}
\begin{center}
\includegraphics[width=\textwidth]{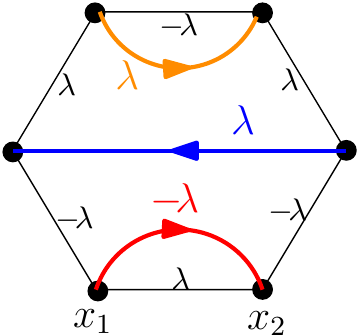}
\end{center}
\caption{\{12,36,45\}}
\end{subfigure}
\begin{subfigure}[b]{0.19\textwidth}
\begin{center}
\includegraphics[width=\textwidth]{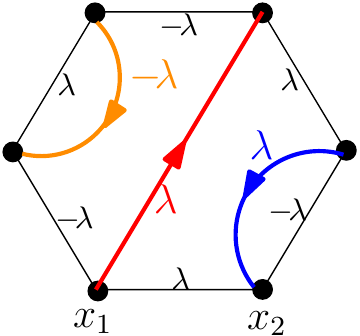}
\end{center}
\caption{\{14,23,56\}}
\end{subfigure}
\begin{subfigure}[b]{0.19\textwidth}
\begin{center}
\includegraphics[width=\textwidth]{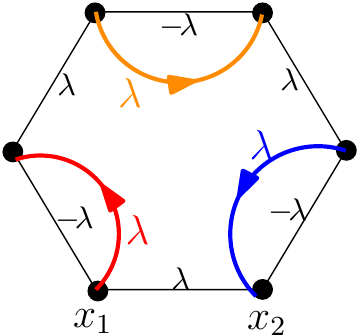}
\end{center}
\caption{\{16,23,45\}}
\end{subfigure}
\begin{subfigure}[b]{0.19\textwidth}
\begin{center}
\includegraphics[width=\textwidth]{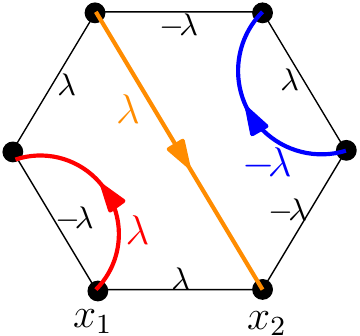}
\end{center}
\caption{\{16,25,34\}}
\end{subfigure}

\caption{\label{fig::threesle} Fix six boundary points $x_1, x_2, ..., x_6$ (in counterclockwise order) and suppose $h$ is a $\GFF$ with alternating boundary conditions: $\lambda$ on $(x_1, x_2)\cup(x_3,x_4)\cup(x_5,x_6)$ and $-\lambda$ elsewhere. Let $\eta_1$ (red) be the level line starting from $x_1$, $\eta_2$ (blue) be the level line starting from $x_3$, and $\eta_3$ (orange) be the level line starting from $x_5$. The end points of $\{\eta_1, \eta_2,\eta_3\}$ form a planar pair partition of $\{1,2,3,4,5,6\}$.}
\end{figure}

\begin{remark}\label{rem::slepair_lessthan4_degenerate}
[Degenerate case in Theorem \ref{thm::slepair_lessthan4}---level lines of $\GFF$]
When $\kappa=4$, the SDE (\ref{eqn::hyperSLE_sde}) degenerates to the SDE (\ref{eqn::sle_rho_sde}).  In other words, $\hSLE_4(\rho)$ in $\HH$ with marked points $(x,y)$ is the same as $\SLE_4(\rho+2,-\rho-2)$ with force points $(x,y)$. In this degenerate case, the pair of curves in Theorem \ref{thm::slepair_lessthan4} can be realized by a pair of level lines of $\GFF$. 
Fix a quad $(\Omega; x^L, y^L, y^R, x^R)$. Let $h$ be the $\GFF$ on $\Omega$ with the following boundary data:
$-\lambda(1+\rho^L)$ on $(y^Lx^L)$, $\lambda$ on $(x^Lx^R)\cup(y^Ry^L)$, and $\lambda(3+\rho^R)$ on $(x^Ry^R)$.
Let $\eta^L$ be the level line of $h$ with height zero starting from $x^L$ and let $\eta^R$ be the level line of $h$ with height $2\lambda$ starting from $x^R$. Then we have that the marginal law of $\eta^L$ is $\SLE_4(\rho^L; \rho^R+2,-\rho^R-2)$ and the conditional law of $\eta^L$ given $\eta^R$ is $\SLE_4(\rho^L)$; the marginal law of $\eta^R$ is $\SLE_4(-\rho^L-2, \rho^L+2; \rho^R)$ and the conditional law of $\eta^R$ given $\eta^L$ is $\SLE_4(\rho^R)$. Thus the pair $(\eta^L;\eta^R)$ is the unique pair described in Theorem \ref{thm::slepair_lessthan4} for $\kappa=4$. 

Furthermore, the case with $\kappa=4$ can be easily generalized to multiple curves. Fix $N\ge 2$, a simply connected proper subset  $\Omega\subset\C$, and $x_1, x_2, ..., x_{2N}$ along the boundary of $\Omega$ in counterclockwise order. For each planar pair partition $\alpha=\{a_1b_1, ..., a_Nb_N\}$ of $\{1,2,3,...,2N\}$, let $X^{\alpha}_0(\Omega; x_1, ..., x_{2N})$ be the collection of non-intersecting simple curves $\{\eta_1, ..., \eta_N\}$ such that $\eta_j\in X_0(\Omega; x_{a_j}, x_{b_j})$ for $1\le j\le N$. Then there exists a unique probability measure on $X^{\alpha}_0(\Omega; x_1, ..., x_{2N})$ with the following property: the conditional law of $\eta_j$ given $\{\eta_1, ..., \eta_{j-1}, \eta_{j+1},...,\eta_N\}$ is $\SLE_4$ process for each $j\in\{1,...,N\}$. 

Moreover, this unique probability measure can be obtained by level lines of $\GFF$. Suppose $h$ is a $\GFF$ in $\Omega$ with alternating boundary conditions: $\lambda$ on $\cup_{1\le j\le N}(x_{2j-1}x_{2j})$ and $-\lambda$ on $\cup_{1\le j\le N}(x_{2j}x_{2j+1})$ (with the convention that $x_{2N+1}=x_1$). For $1\le j\le N$, let $\eta_j$ be the level line of $h$ starting from $x_{2j-1}$. Then the end points of $\{\eta_1, ..., \eta_N\}$ form a planar pair partition $\LA$ of $\{1, 2, ..., 2N\}$, see Figure \ref{fig::threesle}. For each planar pair partition $\alpha$ of $\{1,2,..., 2N\}$, the probability $\PP[\LA=\alpha]$ is strictly positive. Conditioned on the event $\{\LA=\alpha\}$, the collection of level lines $\{\eta_1, ..., \eta_N\}$ is in $X_0^{\alpha}(\Omega; x_1, ..., x_{2N})$, and it has the property that the conditional law of $\eta_j$ given $\{\eta_1, ..., \eta_{j-1}, \eta_{j+1}, \eta_N\}$ is the same as the level line of $\GFF$ with Dobrushin boundary condition, thus the conditional law is $\SLE_4$, for all $j\in\{1,...,N\}$.  
\end{remark}

\subsection{Proof of Proposition \ref{prop::slemultiple_degenerate}}
\begin{lemma} \label{lem::slemultiple}
Fix a Dobrushin domain $(\Omega; x,y)$ and $\kappa\in (0,8)$, $\rho^L>-2$, $\rho^R>-2$ and $\nu\ge\kappa/2-2$. There exists a unique probability measure on pairs of curves $(\eta^L;\eta^R)$ in $X_0^2(\Omega; x, y)$ with the following property: the conditional law of $\eta^R$ given $\eta^L$ is $\SLE_{\kappa}(\nu; \rho^R)$ and the conditional law of $\eta^L$ given $\eta^R$ is $\SLE_{\kappa}(\rho^L;\nu)$. Under this probability measure, the marginal law of $\eta^L$ is $\SLE_{\kappa}(\rho^L; \rho^R+\nu+2)$ and the marginal law of $\eta^R$ is $\SLE_{\kappa}(\rho^L+\nu+2;\rho^R)$.
\end{lemma}
\begin{proof}
By the argument before the proof of Proposition \ref{prop::slepair_uniqueness_general}, we only need to show the conclusion for $\kappa\in (0,4]$. We first show the existence of the pair by checking that a certain pair of flow lines in $\GFF$ satisfies the desired property. Suppose $h$ is a $\GFF$ in $\HH$ with boundary data 
\[-\lambda(2+\rho^L+\nu/2)\text{ on }(-\infty,0),\quad \lambda(2+\rho^R+\nu/2)\text{ on }(0,\infty).\]
Set $\theta=\lambda(\nu/2+1)/\chi>0$. Let $\eta^L$ be the flow line of $h$ with angle $\theta$ and $\eta^R$ be the flow line of $h$ with angle $-\theta$. Then one can check that the conditional law of $\eta^R$ given $\eta^L$ is $\SLE_{\kappa}(\nu; \rho^R)$ and the conditional law of $\eta^L$ given $\eta^R$ is $\SLE_{\kappa}(\rho^L;\nu)$; and that 
the marginal laws of $\eta^L$ and $\eta^R$ are the same as the one in the statement. 
It remains to show the uniqueness. This can be obtained by reducing to the setting that the end points are distinct---Proposition \ref{prop::slepair_uniqueness_general}. 
This is explained at the begining of \cite[Proof of Theorem 4.1]{MillerSheffieldIG2}.
\end{proof}

\begin{proof}[Proof of Proposition \ref{prop::slemultiple_degenerate}]
We first check that the curves $(\eta_1;...;\eta_n)$ can be realized by flow lines of $\GFF$. Suppose $h$ is a $\GFF$ on $\HH$ with the boundary data: $-n\lambda$ on $\R_-$ and $n\lambda$ on $\R_+$. Set $\theta_j=(n+1-2j)\lambda/\chi$ for $j=1,...,n$. Let $\eta_j$ be the flow line of $h$ with angle $\theta_j$. Then one can check that the conditional law of $\eta_j$ given $\eta_{j-1}$ and $\eta_{j+1}$ is $\SLE_{\kappa}$ for $j\in\{1,..., n\}$ and the marginal law of $\eta_j$ is $\SLE_{\kappa}(2j-2;2n-2j)$. 

We prove the uniqueness by induction on $n$. Lemma \ref{lem::slemultiple} implies the conclusion for $n=2$. Suppose the conclusion holds for $n-1$ path and consider $(\eta_1;...;\eta_n)$. Applying the hypothesis to $(\eta_2;...;\eta_n)$, we know the conditional law of $(\eta_2;...;\eta_n)$ given $\eta_1$. In particular, we know that the conditional law of $\eta_2$ given $\eta_1$ is $\SLE_{\kappa}(2n-4)$ and the conditional law of $\eta_1$ given $\eta_2$ is $\SLE_{\kappa}$. Thus, by Lemma \ref{lem::slemultiple}, we know that the marginal law of $\eta_1$ is $\SLE_{\kappa}(2n-2)$. The marginal law of $\eta_1$ and the conditional law of $(\eta_2;...;\eta_n)$ given $\eta_1$ uniquely determine the law of $(\eta_1;...;\eta_n)$. This completes the proof.  
\end{proof}
\begin{remark}
The conclusions in Proposition \ref{prop::slemultiple_degenerate} hold as long as the terminal points of curves coincide. Suppose that $a_1, ..., a_n, b$ are boundary points of $\Omega$ in counterclockwise order and denote by $X_0^n(\Omega; a_1, ..., a_n, b)$ the collection of curves $(\eta_1; ...; \eta_n)$ where $\eta_j\in X_0(\Omega; a_j, b)$ and it is to the right of $\eta_{j-1}$ and is to the left of $\eta_{j+1}$ for $j\in\{1,...,n\}$ with the convention that $\eta_0=(ba_1)$ and $\eta_{n+1}=(a_nb)$. Fix $\kappa\in (0,4]$. Then there is a unique probability measure on $X_0^n(\Omega; a_1, ..., a_n, b)$ such that the the conditional law of $\eta_j$ given $\eta_{j-1}$ and $\eta_{j+1}$ is $\SLE_{\kappa}$ for $j\in\{1,...,n\}$. Under this measure, the marginal law of $\eta_j$ is $\SLE_{\kappa}(\underline{\rho}^L_j;\underline{\rho}^R_j)$ where each of $\{a_1, ..., a_{j-1}, a_{j+1},..., a_n\}$ corresponds to a force point with weight $2$. 
\end{remark}

\section{Ising Model and Proof of Theorem \ref{thm::ising_cvg_pair}}\label{sec::ising}
\textbf{Notations and terminologies.}
We focus on the square lattice $\Z^2$. Two vertices $x=(x_1, x_2)$ and $y=(y_1, y_2)$ are neighbors if $|x_1-y_1|+|x_2-y_2|=1$, and we write $x\sim y$. 
The \textit{dual square lattice} $(\Z^2)^*$ is the dual graph of $\Z^2$. The vertex set is $(1/2, 1/2)+\Z^2$ and the edges are given by nearest neighbors.  The vertices and edges of $(\Z^2)^*$ are called dual-vertices and dual-edges. In particular, for each edge $e$ of $\Z^2$, it is associated to a dual edge, denoted by $e^*$, that it crosses $e$ in the middle. 
For a finite subgraph $G$, we define $G^*$ to be the subgraph of $(\Z^2)^*$ with edge-set $E(G^*)=\{e^*: e\in E(G)\}$ and vertex set given by the end-points of these dual-edges. The \textit{medial lattice} $(\Z^2)^{\diamond}$ is the graph with the centers of edges of $\Z^2$ as vertex set, and edges connecting nearest vertices. This lattice is a rotated and rescaled version of $\Z^2$, see Figure \ref{fig::medial_lattice}. The vertices and edges of $(\Z^2)^{\diamond}$ are called medial-vertices and medial-edges. We identify the faces of $(\Z^2)^{\diamond}$ with the vertices of $\Z^2$ and $(\Z^2)^*$. A face of $(\Z^2)^{\diamond}$ is said to be black if it corresponds to a vertex of $\Z^2$ and white if it corresponds to a vertex of $(\Z^2)^*$. 

\begin{figure}[ht!]
\begin{subfigure}[b]{0.33\textwidth}
\begin{center}
\includegraphics[width=0.8\textwidth]{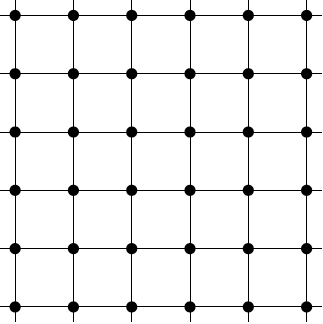}
\end{center}
\caption{The square lattice. }
\end{subfigure}
\begin{subfigure}[b]{0.33\textwidth}
\begin{center}\includegraphics[width=0.8\textwidth]{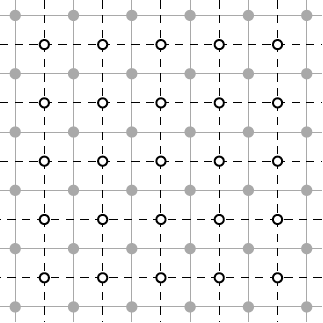}
\end{center}
\caption{The dual square lattice.}
\end{subfigure}
\begin{subfigure}[b]{0.33\textwidth}
\begin{center}
\includegraphics[width=0.8\textwidth]{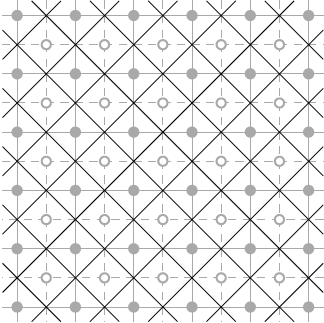}
\end{center}
\caption{The medial lattice. }
\end{subfigure}
\caption{\label{fig::medial_lattice} The lattices. }
\end{figure}

\subsection{Ising model}
Let $\Omega$ be a finite subset of $\Z^2$. The Ising model with free boundary conditions is a random assignment $\sigma\in \{\ominus, \oplus\}^{\Omega}$ of spins $\sigma_x\in \{\ominus, \oplus\}$, where $\sigma_x$ denotes the spin at the vertex $x$. The Hamiltonian of the Ising model is defined by 
\[H^{\free}_{\Omega}(\sigma)=-\sum_{x\sim y}\sigma_x\sigma_y.\] 
The Ising measure is the Boltzmann measure with Hamiltonian $H^{\free}_{\Omega}$ and inverse-temperature $\beta>0$: 
\[\mu^{\free}_{\beta,\Omega}[\sigma]=\frac{\exp(-\beta H^{\free}_{\Omega}(\sigma))}{Z^{\free}_{\beta, \Omega}},\quad\text{where }Z^{\free}_{\beta, \Omega}=\sum_{\sigma}\exp(-\beta H^{\free}_{\Omega}(\sigma)).\]

For a graph $\Omega$ and $\tau\in \{\ominus, \oplus\}^{\Z^2}$, one may also define the Ising model with boundary conditions $\tau$ by the Hamiltonian
\[H^{\tau}_{\Omega}(\sigma)=-\sum_{x\sim y, \{x,y\}\cap \Omega\neq\emptyset}\sigma_x\sigma_y,\quad \text{if }\sigma_x=\tau_x, \forall x\not\in \Omega.\]
Ising model satisfies Domain Markov property: 
Let $\Omega\subset \Omega'$ be two finite subsets of $\Z^2$. Let $\tau\in \{\ominus, \oplus\}^{\Z^2}$ and $\beta>0$. Let $X$ be a random variable which is measurable with respect to vertices in $\Omega$. Then we have 
\[\mu^{\tau}_{\beta, \Omega'}[X\cond \sigma_x=\tau_x, \forall x\in \Omega'\setminus \Omega]=\mu^{\tau}_{\beta, \Omega}[X].\] 

Suppose that $(\Omega; a, b)$ is a Dobrushin domain. The \textit{Dobrushin boundary condition} is the following: $\oplus$ along $(ab)$, and $\ominus$ along $(ba)$. This boundary condition is also called domain-wall boundary condition. Suppose that $(\Omega; a, b, c, d)$ is a quad. The \textit{alternating boundary condition} is the following: $\oplus$ along $(ab)$ and $(cd)$, and $\ominus$ along $(bc)$ and $(da)$. 

The set $\{\ominus, \oplus\}^{\Omega}$ is equipped with a partial order: $\sigma\le \sigma'$ if $\sigma_x\le \sigma_x'$ for all $x\in \Omega$. A random variable $X$ is increasing if $\sigma\le \sigma'$ implies $X(\sigma)\le X(\sigma')$. An event $\LA$ is increasing if $\one_{\LA}$ is increasing. The Ising model satisfies FKG inequality:
Let $\Omega$ be a finite subset and $\tau$ be boundary conditions, and $\beta>0$. For any two increasing events $\LA$ and $\LB$, we have 
\[\mu^{\tau}_{\beta, \Omega}[\LA\cap\LB]\ge \mu^{\tau}_{\beta, \Omega}[\LA]\mu^{\tau}_{\beta, \Omega}[\LB].\]
As a consequence of FKG inequality, we have the comparison between boundary conditions: For boundary conditions $\tau_1\le \tau_2$ and an increasing event $\LA$, we have
\begin{equation}\label{eqn::ising_boundary_comparison}
\mu^{\tau_1}_{\beta, \Omega}[\LA]\le \mu^{\tau_2}_{\beta, \Omega}[\LA]. 
\end{equation}

The critical value of $\beta$ in Ising model is given by:
\[\beta_c=\frac{1}{2}\log(1+\sqrt{2}).\] 
The critical Ising model is conformal invariant in the scaling limit. We will list two special properties for the critical Ising model that will be useful later: strong RSW and the convergence of the interface with Dobrushin boundary condition. 

Given a quad $(Q; a, b, c, d)$ on the square lattice, we denote by $d_{Q}((ab), (cd))$ the discrete extermal distance between $(ab)$ and $(cd)$ in $Q$, see \cite[Section 6]{ChelkakRobustComplexAnalysis}. The discrete extremal distance is uniformly comparable to and converges to its continuous counterpart--- the classical extremal distance. The quad $(Q; a, b, c, d)$ is crossed by $\oplus$ in an Ising configuration $\sigma$ if there exists a path of $\oplus$ going from $(ab)$ to $(cd)$ in $Q$. We denote this event by $(ab)\overset{\oplus}{\longleftrightarrow}(cd)$. 
\begin{proposition}
\label{prop::ising_rsw}\cite[Corollary 1.7]{ChelkakDuminilHonglerCrossingprobaFKIsing}.
For each $L>0$ there exists $c(L)>0$ such that the following holds: for any quad $(Q; a, b, c, d)$ with $d_{Q}((ab), (cd))\ge L$, 
\[\mu^{\text{mixed}}_{\beta_c, Q}\left[(ab)\overset{\oplus}{\longleftrightarrow}(cd)\right]\le 1-c(L),\]
where the boundary conditions are $\free$ on $(ab)\cup(cd)$ and $\ominus$ on $(bc)\cup(da)$. 
\end{proposition}

\begin{figure}[ht!]
\begin{center}
\includegraphics[width=0.8\textwidth]{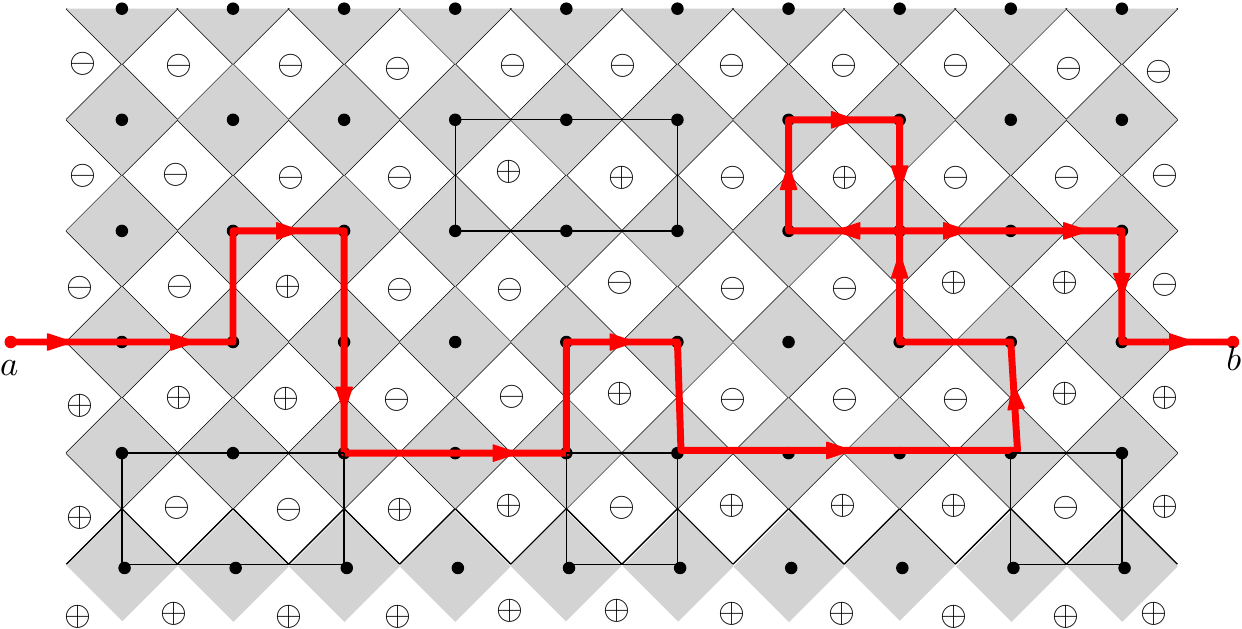}
\end{center}
\caption{\label{fig::ising_interface} The Ising interface with Dobrushin boundary condition. }
\end{figure}

For $\delta>0$, we consider the rescaled square lattice $\delta\Z^2$. The definitions of dual lattice, medial lattice and  Dobrushin domains extend to this context, and they will be denoted by $(\Omega_{\delta}, a_{\delta}, b_{\delta})$, $(\Omega^*_{\delta}, a^*_{\delta}, b^*_{\delta})$, $(\Omega^{\diamond}_{\delta}, a^{\diamond}_{\delta}, b^{\diamond}_{\delta})$ respectively. 
Consider the critical Ising model on $(\Omega^*_{\delta}, a^*_{\delta}, b^*_{\delta})$. The boundary $\partial \Omega^*_{\delta}$ is divided into two parts $(a^*_{\delta}b^*_{\delta})$ and $(b^*_{\delta}a^*_{\delta})$. We fix the Dobrushin boundary conditions: $\ominus$ on $(b^*_{\delta}a^*_{\delta})$ and $\oplus$ on $(a^*_{\delta}b^*_{\delta})$. Define the \textit{interface} as follows. It starts from $a_{\delta}^{\diamond}$, lies on the primal lattice and turns at every vertex of $\Omega_{\delta}$ is such a way that it has always dual vertices with spin $\ominus$ on its left and $\oplus$ on its right. If there is an indetermination when arriving at a vertex (this may happen on the square lattice), turn left. See Figure \ref{fig::ising_interface}. We have the convergence of the interface:  

\begin{theorem}\label{thm::ising_cvg_minusplus}\cite{CDCHKSConvergenceIsingSLE}.
Let $(\Omega^{\diamond}_{\delta}; a^{\diamond}_{\delta}, b^{\diamond}_{\delta})$ be a family of Dobrushin domains converging to a Dobrushin domain $(\Omega; a, b)$ in the Carath\'eodory sense. The interface of the critical Ising model in $(\Omega^*_{\delta}, a^*_{\delta}, b^*_{\delta})$ with Dobrushin boundary condition converges weakly to $\SLE_{3}$ as $\delta\to 0$.
\end{theorem}

\begin{theorem}\label{thm::ising_cvg_minusfree}
Let $(\Omega^{\diamond}_{\delta}; a^{\diamond}_{\delta}, b^{\diamond}_{\delta})$ be a family of Dobrushin domains converging to a Dobrushin domain $(\Omega; a, b)$ in the Carath\'eodory sense. The interface of the critical Ising model in $(\Omega^*_{\delta}, a^*_{\delta}, b^*_{\delta})$ with the boundary condition $\ominus$ along $(b^*_{\delta}a^*_{\delta})$ and $\free$ along $(a^*_{\delta}b^*_{\delta})$ converges weakly to $\SLE_{3}(-3/2)$ as $\delta\to 0$.
\end{theorem}
\begin{proof}
It is proved in \cite{HonglerKytolaIsingFree, BenoistDuminilHonglerIsingFree} that the interface with $(\free\free)$ boundary conditions converges weakly to $\SLE_3(-3/2;-3/2)$ as $\delta\to 0$. The same proof works here. 
\end{proof}

\subsection{Proof of Theorem \ref{thm::ising_cvg_pair}}
Let $(\Omega_{\delta}; x^L_{\delta}, x^R_{\delta}, y^R_{\delta}, y^L_{\delta})$ be a sequence of discrete quads on the square lattice $\delta\Z^2$ approximating some quad $(\Omega; x^L, x^R, y^R, y^L)$. Consider the critical Ising model in $\Omega_{\delta}^*$ with alternating boundary conditions: $\ominus$ along $(x_{\delta}^Lx^R_{\delta})$ and $(y_{\delta}^Ry_{\delta}^L)$, and $\xi^R\in\{\oplus, \free\}$ along $(x_{\delta}^Ry_{\delta}^R)$, and $\xi^L\in\{\oplus,\free\}$ along $(y_{\delta}^Lx_{\delta}^L)$. Suppose there is a vertical crossing of $\ominus$ and denote this event by
\[\LC^{\ominus}_v(\Omega_{\delta})=\{(x^L_{\delta}x^R_{\delta})\overset{\ominus}{\longleftrightarrow}(y^R_{\delta}y^L_{\delta})\}.\]
Let $\eta^L_{\delta}$ be the interface starting from $x^L_{\delta}$ lying on the primal lattice. It turns at every vertex in the way that it has spin $\oplus$ on its left and $\ominus$ on its right, and that it turns left when there is ambiguity. 
Let $\eta^R_{\delta}$ be the interface starting from $x^R_{\delta}$ lying on the primal lattice. It turns at every vertex in the way that it has spin $\ominus$ to its left and $\oplus$ to its right, and turns right when there is ambiguity. Then $\eta^L_{\delta}$ will end at $y^L_{\delta}$ and $\eta^R_{\delta}$ will end at $y^R_{\delta}$. See Figure \ref{fig::ising_pair}. Let $\Omega^L_{\delta}$ be the connected component of $\Omega_{\delta}\setminus\eta^L_{\delta}$ with $(x_{\delta}^Ry_{\delta}^R)$ on the boundary and denote by $\LD^L_{\delta}$ the discrete extremal distance between $\eta^L_{\delta}$ and $(x^R_{\delta}y^R_{\delta})$ in $\Omega^L_{\delta}$. Define $\Omega^R_{\delta}$ and $\LD^R_{\delta}$ similarly. 
\begin{lemma}\label{lem::ising_distance_tight}
The variables $(\LD^L_{\delta}; \LD^R_{\delta})_{\delta>0}$ is tight in the following sense: for any $u>0$, there exists $\eps>0$ such that
\[\PP\left[\LD^L_{\delta}\ge\eps, \LD^R_{\delta}\ge \eps\cond \LC^{\ominus}_v(\Omega_{\delta}) \right]\ge 1-u,\quad \forall \delta>0.\]
\end{lemma}
\begin{proof}
Since $(\Omega_{\delta}; x^L_{\delta}, x^R_{\delta}, y^R_{\delta}, y^L_{\delta})$ approximates $(\Omega; x^L, x^R, y^R, y^L)$, 
by Proposition \ref{prop::ising_rsw} and (\ref{eqn::ising_boundary_comparison}), we know that $\PP[\LC_v^{\ominus}(\Omega_{\delta})]$ can be bounded from below by some quantity that depends only on the extremal distance in $\Omega$ between $(x^Lx^R)$ and $(y^Ry^L)$ and that is uniform over $\delta$. 
Thus, it is sufficient to show $\PP\left[\{\LD^L_{\delta}\le\eps\}\cap\LC_v^{\ominus}(\Omega_{\delta})\right]$ is small for $\eps>0$ small. Given $\eta^L_{\delta}$ and on the event $\{\LD^L_{\delta}\le\eps\}$, combining Proposition \ref{prop::ising_rsw} and (\ref{eqn::ising_boundary_comparison}), we know that the probability to have a vertical crossing of $\ominus$ in $\Omega^L_{\delta}$ is bounded by $c(\eps)$ which only depends on $\eps$ and goes to zero as $\eps\to 0$. Thus $\PP\left[\{\LD^L_{\delta}\le\eps\}\cap\LC_v^{\ominus}(\Omega_{\delta})\right]\le c(\eps)$. This implies the conclusion. 
\end{proof}
\begin{proof}[Proof of Theorem \ref{thm::ising_cvg_pair}]
We only prove the conclusion for $\xi^R=\xi^L=\oplus$, and the other cases can be proved similarly (by replacing Theorem \ref{thm::ising_cvg_minusplus} by \ref{thm::ising_cvg_minusfree} when necessary). 
Combining Proposition \ref{prop::ising_rsw} with (\ref{eqn::ising_boundary_comparison}) and Lemma \ref{lem::ising_distance_tight}, we see that the sequence $\{(\eta^L_{\delta};\eta^R_{\delta})\}_{\delta>0}$ satisfies the requirements in Theorem \ref{thm::cvg_pairs}, thus the sequence is relatively compact. Suppose $(\eta^L;\eta^R)\in X_0(\Omega; x^L, x^R, y^R, y^L)$ is any sub-sequential limit and, for some $\delta_k\to 0$, 
\[(\eta^L_{\delta_k};\eta^R_{\delta_k})\overset{d}{\longrightarrow}(\eta^L;\eta^R)\quad \text{ in } X_0(\Omega; x^L, x^R, y^R, y^L).\]
Since $\eta^L_{\delta_k}\to \eta^L$, by Theorem \ref{thm::cvg_curves_chordal}, we know that we have the convergence in all three topologies. In particular, this implies the convergence of $\Omega^L_{\delta_k}$ in Carath\'eodory sense. Note that the conditional law of $\eta^R_{\delta_k}$ in $\Omega^L_{\delta_k}$ given $\eta^L_{\delta_k}$ is the interface of the critical planar Ising model with Dobrushin boundary condition. Combining with Theorem \ref{thm::ising_cvg_minusplus}, we know that, the conditional law of $\eta^R$ in $\Omega^L$ given $\eta^L$ is $\SLE_3$. By symmetry, the conditional law of $\eta^L$ in $\Omega^R$ given $\eta^R$ is $\SLE_3$. By Theorem \ref{thm::slepair_lessthan4}, there exists a unique such measure. Thus it has to be the unique sub-sequential limit. This completes the proof of Theorem \ref{thm::ising_cvg_pair}.
\end{proof}

\section{FK-Ising Model and Proof of Proposition \ref{prop::fkising_cvg_pair}}\label{sec::fkising}
\subsection{FK-Ising model}
We will consider finite subgraphs $G=(V(G), E(G))\subset \Z^2$. For such a graph, we denote by $\partial G$ the inner boundary of $G$:
\[\partial G=\{x\in V(G): \exists y\not\in V(G) \text{ such that } \{x,y\}\in E(\Z^2)\}.\]  

A \textit{configuration} $\omega=(\omega_e: e\in E(G))$ is an element of $\{0,1\}^{E(G)}$. If $\omega_e=1$, the edge $e$ is said to be open, otherwise $e$ is said to be closed. The configuration $\omega$ can be seen as a subgraph of $G$ with the same set of vertices $V(G)$, and the set of edges given by open edges $\{e\in E(G): \omega_e=1\}$.  

We are interested in the connectivity properties of the graph $\omega$. The maximal connected components of $\omega$ are called clusters. Two vertices $x$ and $y$ are connected by $\omega$ inside $S\subset \Z^2$ if there exists a path of vertices $(v_i)_{0\le i\le k}$  in $S$ such that $v_0=x, v_k=y$ and $\{v_i, v_{i+1}\}$ is open in $\omega$ for $0\le i<k$. We denote this event by $\{x\overset{S}{\longleftrightarrow}y\}$. If $S=G$, we simply drop it from the notation. For $A, B\subset \Z^2$, set $\{A\overset{S}{\longleftrightarrow} B\}$ if there exists a vertex of $A$ connected in $S$ to a vertex in $B$. 

Given a finite subgraph $G\subset \Z^2$, \textit{boundary condition} $\xi$ is a partition $P_1\sqcup \cdots\sqcup P_k$ of $\partial G$. Two vertices are wired in $\xi$ if they belong to the same $P_i$. The graph obtained from the configuration $\omega$ by identifying the wired vertices together in $\xi$ is denoted by $\omega^{\xi}$. Boundary conditions should be understood informally as encoding how sites are connected outside of $G$. Let $o(\omega)$ and $c(\omega)$ denote the number of open can dual edges of $\omega$ and $k(\omega^{\xi})$ denote the number of maximal connected components of the graph $\omega^{\xi}$.  

The probability measure $\phi^{\xi}_{p,q,G}$ of the \textit{random cluster model} model on $G$ with edge-weight $p\in [0,1]$, cluster-weight $q>0$ and boundary condition $\xi$ is defined by 
\[\phi^{\xi}_{p,q,G}[\omega]:=p^{o(\omega)}(1-p)^{c(\omega)}q^{k(\omega^{\xi})}/Z^{\xi}_{p,q,G},\]
where $Z^{\xi}_{p,q,G}$ is the normalizing constant to make $\phi^{\xi}_{p,q,G}$ a probability measure. For $q=1$, this model is simply Bernoulli bond percolation. 

For a configuration $\xi$ on $E(\Z^2)\setminus E(G)$, the boundary conditions induced by $\xi$ are defined by the partition $P_1\sqcup \cdots \sqcup P_k$, where $x$ and $y$ are in the same $P_i$ if and only if there exists an open path in $\xi$ connecting $x$ and $y$. We identify the boundary condition induced by $\xi$ with the configuration itself, and denote the random cluster model with these boundary conditions by $\phi^{\xi}_{p,q,G}$. As a direct consequence of these definitions, we have the Domain Markov Property of the random cluster model: 
Suppose that $G'\subset G$ are two finite subgraphs of $\Z^2$. 
Fix $p\in [0,1], q>0$ and $\xi$ some boundary conditions on $\partial G$. Let $X$ be a random variable which is measurable with respect to edges in $E(G')$. Then we have
\[\phi^{\xi}_{p,q,G}\left[X\cond \omega_e=\psi_e, \forall e\in E(G)\setminus E(G')\right]=\phi^{\psi^{\xi}}_{p,q,G}[X],\quad \forall \psi\in \{0,1\}^{E(G)\setminus E(G')},\]
where $\psi^{\xi}$ is the partition on $\partial G'$ obtained as follows: two vertices $x,y\in\partial G'$ are wired if they are connected in $\psi^{\xi}$. 

Suppose that $(\Omega; a, b)$ is a Dobrushin domain. The Dobrushin boundary condition is the following: free along $(ab)$ and wired along $(ba)$. Suppose that $(\Omega; a, b, c, d)$ is a quad. The alternating boundary condition is the following: free along $(ab)$ and $(cd)$, wired along $(bc)$ and $(da)$. 

Denote the product ordering on $\{0,1\}^E$ by $\le$. In other words, for $\omega, \omega'\in \{0,1\}^E$, we denote by $\omega\le \omega'$ if $\omega_e\le \omega'_e$, for all $e\in E$. An event $\LA$ depending on edges in $E$ is \textit{increasing} if for any $\omega\in \LA, \omega\le \omega'$ implies $\omega'\in \LA$. We have positive association (FKG inequality) when $q\ge 1$: 
Fix $p\in [0,1], q\ge 1$ and a finite graph $G$ and some boundary conditions $\xi$. For any two increasing events $\LA$ and $\LB$, we have 
\[\phi^{\xi}_{p,q,G}[\LA\cap \LB]\ge \phi^{\xi}_{p,q,G}[\LA]\phi^{\xi}_{p,q,G}[\LB].\]
 
As a consequence of the FKG inequality, we have the comparison principle between boundary conditions: fix $p\in [0,1], q\ge 1$ and a finite graph $G$. For any boundary conditions $\xi\le\psi$ and any increasing event $\LA$, we have 
\[\phi^{\xi}_{p,q,G}[\LA]\le \phi^{\psi}_{p,q,G}[\LA]. \]

For a finite subgraph $G$, we define $G^*$ to be the subgraph of $(\Z^2)^*$ with edge-set $E(G^*)=\{e^*: e\in E(G)\}$ and vertex set given by the end-points of these dual-edges. A configuration $\omega$ on $G$ can be uniquely associated to a dual configuration $\omega^*$ on the dual graph $G^*$ defined as follows: set $\omega^*(e^*)=1-\omega(e)$ for all $e\in E(G)$. A dual-edge $e^*$ is said to be dual-open if $\omega^*(e^*)=1$, it is dual-closed otherwise. A dual-cluster is a connected component of $\omega^*$. We extend the notion of dual-open path and the connective events in the obvious way. 
If $\omega$ is distributed according to $\phi^{\xi}_{p,q,G}$, then $\omega^*$ is distributed according to $\phi^{\xi^*}_{p^*, q^*, G^*}$ where 
\[q^*=q, \quad \frac{pp^*}{(1-p)(1-p^*)}=q,\]
and the boundary conditions $\xi^*$ can be deduced from $\xi$ in a case by case manner. In particular, $\xi=0$ corresponds to $\xi^*=1$ and $\xi=1$ corresponds to $\xi^*=0$.

The critical value of $p$ for a given $q$ is the following:
\[p_c(q)=\frac{\sqrt{q}}{1+\sqrt{q}}.\]
People believe that the critical random-cluster model is conformal invariant in the scaling limit for $q\in [1,4]$, and it is only proved for $q=2$ in \cite{ChelkakSmirnovIsing, CDCHKSConvergenceIsingSLE}. When $q=2$, the critical random-cluster model is also called FK-Ising model. 
We will list two special properties for the FK-Ising model that will be useful later: strong RSW and the convergence of the interface with Dobrushin boundary condition. 

\begin{proposition}\label{prop::rcm_rsw_strong}
\cite[Theorem 1.1]{ChelkakDuminilHonglerCrossingprobaFKIsing}.
For each $L>0$ there exists $c(L)>0$ such that the following holds: for any quad $(Q, a, b, c, d)$ with $d_Q((ab), (cd))\ge L$ and for any boundary conditions $\xi$, 
\[\phi^{\xi}_{p_c(2), 2, \Omega}\left[(ab)\leftrightarrow(cd)\right]\le 1- c(L).\]
\end{proposition}

\begin{figure}[ht!]
\begin{subfigure}[b]{0.5\textwidth}
\begin{center}
\includegraphics[width=\textwidth]{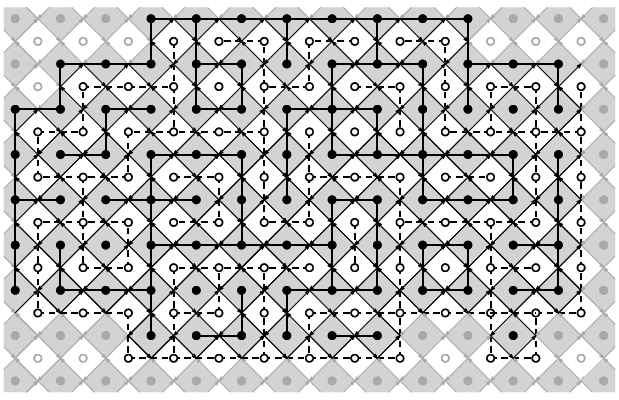}
\end{center}
\caption{The configuration $\omega$ and its dual $\omega^*$. }
\end{subfigure}
\begin{subfigure}[b]{0.5\textwidth}
\begin{center}\includegraphics[width=\textwidth]{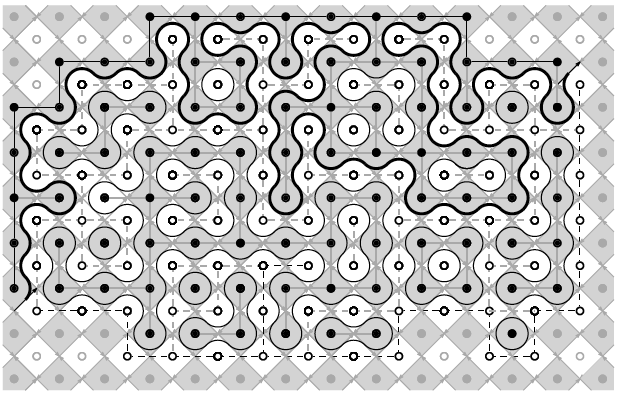}
\end{center}
\caption{The loop representation of $\omega$.}
\end{subfigure}
\caption{\label{fig::loop_representation} The loop representation of the configuration. }
\end{figure}

Fix a Dobrushin domain $(\Omega; a, b)$ and consider a configuration $\omega$ together with its dual-configuration $\omega^*$. The Dobrushin boundary condition is given by taking edges of $\partial_{ba}$ to be open and the dual-edges of $\partial_{ab}^*$ to be dual-open. Through every vertex of $\Omega^{\diamond}$, there passes either an open edge of $\Omega$ or a dual open edge of $\Omega^*$. Draw self-avoiding loops on $\Omega^{\diamond}$ as follows: a loop arriving at a vertex of the medial lattice always makes a $\pm\pi/2$ turn so as not to cross the open or dual open edges through this vertex, see Figure \ref{fig::loop_representation}. The loop representation contains loops together with a self-avoiding path going from $a^{\diamond}$ to $b^{\diamond}$. This curve is called the \textit{exploration path}. See \cite[Section 6.1]{CDParafermionic} for more details.

\begin{theorem}\label{thm::fkising_cvg}\cite{CDCHKSConvergenceIsingSLE}.
Let $(\Omega_{\delta}; a_{\delta}, b_{\delta})$ be a family of Dobrushin domains converging to a Dobrushin domain $(\Omega; a, b)$ in the Carath\'eodory sense. The exploration path of the critical FK-Ising model with Dobrushin boundary conditions converges weakly to $\SLE_{16/3}$ as $\delta\to 0$.
\end{theorem}

\subsection{Proof of Proposition \ref{prop::fkising_cvg_pair}}
\begin{lemma}\label{lem::fkising_distance_tight}
Let $(\Omega_{\delta}; a_{\delta}, b_{\delta}, c_{\delta}, d_{\delta})$ be a family of quads converging to a quad $(\Omega; a, b, c, d)$ in the Carath\'{e}odory sense. Consider the critical FK-Ising model with Dobrushin boundary condition: wired on $(d_{\delta}a_{\delta})$ and free on $(a_{\delta}d_{\delta})$ and denote by $\eta_{\delta}$ the exploration path from $a_{\delta}$ to $d_{\delta}$. 
Denote by $\LC_v^0(\Omega_{\delta})$ the event that there is a dual-crossing in $\Omega_{\delta}$ from $(a_{\delta}b_{\delta})$ to $(c_{\delta}d_{\delta})$. Then the law of $\eta_{\delta}$ conditioned on the event $\LC_v^0(\Omega_{\delta})$ converges weakly to $\SLE_{16/3}$ in $\Omega$ from $a$ to $d$ conditioned not to hit $(bc)$ as $\delta\to 0$. 
Given $\eta_{\delta}$ and on the event $\LC_v^0(\Omega_{\delta})$, define $\LD_{\delta}$ to be the discrete extremal distance between $\eta_{\delta}$ and $(b_{\delta}c_{\delta})$ in $\Omega^{\eta}_{\delta}$ which is the connected component of $\Omega_{\delta}\setminus\eta_{\delta}$ with $(b_{\delta}c_{\delta})$ on the boundary. Then $\{\LD_{\delta}\}_{\delta>0}$ is tight in the following sense: for any $u>0$, there exists $\eps>0$ such that 
\[\PP[\LD_{\delta}\ge \eps]\ge 1-u,\quad \forall \delta>0.\]
\end{lemma}
\begin{proof}
The convergence is a direct consequence of Theorem \ref{thm::fkising_cvg}. The tightness can be proved by the same argument as in the proof of Lemma \ref{lem::ising_distance_tight} where we need to replace Proposition \ref{prop::ising_rsw} by Proposition \ref{prop::rcm_rsw_strong}.   
\end{proof}
\begin{lemma}\label{lem::fkising_pair_unique}
Let $(\Omega_{\delta}; x^L_{\delta}, x^R_{\delta}, y^R_{\delta}, y^L_{\delta})$ be a sequence of discrete quads on the square lattice $\delta\Z^2$ approximating some quad $(\Omega; x^L, x^R, y^R, y^L)$. Consider the critical FK-Ising model in $\Omega_{\delta}$ with alternating boundary conditions: free on $(x_{\delta}^Lx_{\delta}^R)$ and $(y_{\delta}^Ry_{\delta}^L)$, and wired on $(x_{\delta}^Ry_{\delta}^R)$ and $(y_{\delta}^Lx_{\delta}^L)$.
Conditioned on the event that there are two disjoint vertical dual-crossings, then there exists a pair of interfaces $(\eta^L_{\delta};\eta^R_{\delta})$ where $\eta^L_{\delta}$ (resp. $\eta^R_{\delta}$) is the interface connecting $x^L_{\delta}$ to $y^L_{\delta}$ (resp. connecting $x^R_{\delta}$ to $y^R_{\delta}$).
The law of $(\eta^L_{\delta}; \eta^R_{\delta})$ converges weakly to the unique pair of curves $(\eta^L;\eta^R)$ in $X_0(\Omega; x^L, x^R, y^R, y^L)$ with the following property: Given $\eta^L$, the conditional law of $\eta^R$ is an $\SLE_{16/3}$ conditioned not to hit $\eta^L$; given $\eta^R$, the conditional law of $\eta^L$ is an $\SLE_{16/3}$ conditioned not to hit $\eta^R$.
\end{lemma}
\begin{proof}
Combining Proposition \ref{prop::rcm_rsw_strong} and Lemma \ref{lem::fkising_distance_tight}, we see that the sequence $\{(\eta^L_{\delta};\eta^R_{\delta})\}_{\delta>0}$ satisfies the requirements in Theorem \ref{thm::cvg_pairs}, thus the sequence is relatively compact. 
Suppose $(\eta^L;\eta^R)$ is any subsequential limit and, for some $\delta_k\to 0$, we have $(\eta^L_{\delta_k}; \eta^R_{\delta_k})\to (\eta^L;\eta^R)$ in distribution. 
Let $\Omega^L_{\delta}$ be the connected component of $\Omega\setminus \eta^L_{\delta}$ with $(x^R_{\delta}y^R_{\delta})$ on the boundary, and define $\Omega^L$ similarly. Since $\eta^L_{\delta_k}\to \eta^L$, combining with Theorem \ref{thm::cvg_curves_chordal}, we have the convergence in all three topologies. In particular, the quads $\Omega^L_{\delta_k}$ converges to $\Omega^L$ in Carath\'eodory sense. Combining with Lemma \ref{lem::fkising_distance_tight}, we know that the conditional law of $\eta^R$ given $\eta^L$ is $\SLE_{16/3}$ conditioned not to hit $\eta^L$. By symmetry, the conditional of $\eta^L$ given $\eta^R$ is $\SLE_{16/3}$ conditioned not to hit $\eta^R$. It remains to explain that there is a unique measure on pairs $(\eta^L;\eta^R)$ with such property. The uniqueness can be proved by the same argument as in the proof of Proposition \ref{prop::slepair_uniqueness_general}. The key point is that the two curves $\eta^L, \eta^R$ do not hit each other and that the counterflow line is deterministic function of the $\GFF$ and thus the event that the counterflow line does not hit part of the boundary is also deterministic of the $\GFF$. 
\end{proof}

\begin{proof}[Proof of Proposition \ref{prop::fkising_cvg_pair}]
Fix $\kappa=16/3$. We derive the conclusion by three steps: first, let the quads $(\Omega_{\delta}; x^L_{\delta}, x^R_{\delta}, y^R_{\delta}, y^L_{\delta})$ approximate some quad $(\Omega; x^L, x^R, y^R, y^L)$; second, let $y^L, y^R\to y$; thirdly, let $x^L, x^R\to x$. In the first step, by Lemma \ref{lem::fkising_pair_unique}, we know that the pairs of interfaces converge weakly to a unique probability measure on $(\eta^L; \eta^R)\in X_0(\Omega; x^L, x^R, y^R, y^L)$ such that the conditional law of $\eta^R$ given $\eta^L$ is $\SLE_{\kappa}$ conditioned not to hit $\eta^L$ and the conditional law of $\eta^L$ given $\eta^R$ is $\SLE_{\kappa}$ conditioned not to hit $\eta^R$. Now, let us fix $\eta^L$ and let $y^R\to y:=y^L$. By Lemma \ref{lem::sle_conditioned}, we know that the conditional law of $\eta^R$ given $\eta^L$ converges weakly to $\SLE_{\kappa}(\kappa-4)$. Let $X_0^2(x^L, x^R, y)$ be the collection of pairs of curves $(\eta^L; \eta^R)$ such that $\eta^L\in X_0(\Omega; x^L, y), \eta^R\in X_0(\Omega; x^R, y)$ and that $\eta^L$ is to the left of $\eta^R$. From the above analysis, by sending $y^L, y^R\to y$, the limiting probability measure on $(\eta^L; \eta^R)\in X_0^2(x^L, x^R, y)$ satisfies the following property: the conditional law of $\eta^L$ given $\eta^R$ is $\SLE_{\kappa}(\kappa-4)$ and the conditional law of $\eta^R$ given $\eta^L$ is $\SLE_{\kappa}(\kappa-4)$. By Lemma \ref{lem::slemultiple}, there exists a unique such measure, and under this measure, the marginal law of $\eta^L$ is $\SLE_{\kappa}(\kappa-2)$ with force point located at $x^R$ and the marginal law of $\eta^R$ is $\SLE_{\kappa}(\kappa-2)$ with force point located at $x^L$. Finally, since the law of $\SLE_{\kappa}(\rho)$ process is continuous in the locations of the force points, we complete the proof by sending $x^L, x^R\to x$. 
\end{proof}


\section*{Acknowledgment}

The author is supported by the NCCR/SwissMAP, the ERC AG COMPASP, the Swiss NSF. The author thanks Dmitry Chelkak and Stanislav Smirnov for helpful discussion about Ising and FK-Ising model with alternating boundary conditions. The author thanks Gregory Lawler for helpful discussion on multiple SLEs.

\bibliographystyle{alpha}
\bibliography{bibliography}

\newcommand{\etalchar}[1]{$^{#1}$}
\begin{thebibliography}{CDCH{\etalchar{+}}14}

\bibitem[AS92]{AbramowitzHandbook}
Milton Abramowitz and Irene~A. Stegun, editors.
\newblock {\em Handbook of mathematical functions with formulas, graphs, and
  mathematical tables}.
\newblock Dover Publications, Inc., New York, 1992.
\newblock Reprint of the 1972 edition.

\bibitem[BBK05]{KytolaMultipleSLE}
Michel Bauer, Denis Bernard, and Kalle Kyt{\"o}l{\"a}.
\newblock Multiple {S}chramm-{L}oewner evolutions and statistical mechanics
  martingales.
\newblock {\em J. Stat. Phys.}, 120(5-6):1125--1163, 2005.

\bibitem[BDCH14]{BenoistDuminilHonglerIsingFree}
St{\'e}phane Benoist, Hugo Duminil-Copin, and Cl{\'e}ment Hongler.
\newblock Conformal invariance of crossing probabilities for the ising model
  with free boundary conditions.
\newblock {\em arXiv preprint arXiv:1410.3715}, 2014.

\bibitem[Bil99]{BillingsleyConvergenceProbabilityMeasures}
Patrick Billingsley.
\newblock {\em Convergence of Probability Measures}.
\newblock 1999.

\bibitem[CDCH{\etalchar{+}}14]{CDCHKSConvergenceIsingSLE}
Dmitry Chelkak, Hugo Duminil-Copin, Cl{\'e}ment Hongler, Antti Kemppainen, and
  Stanislav Smirnov.
\newblock Convergence of ising interfaces to schrammʼs sle curves.
\newblock {\em Comptes Rendus Mathematique}, 352(2):157--161, 2014.

\bibitem[CDCH16]{ChelkakDuminilHonglerCrossingprobaFKIsing}
Dmitry Chelkak, Hugo Duminil-Copin, and Cl{\'e}ment Hongler.
\newblock Crossing probabilities in topological rectangles for the critical
  planar {FK}-{I}sing model.
\newblock {\em Electron. J. Probab.}, 21:Paper No. 5, 28, 2016.

\bibitem[Che16]{ChelkakRobustComplexAnalysis}
Dmitry Chelkak.
\newblock Robust discrete complex analysis: A toolbox.
\newblock {\em The Annals of Probability}, 44(1):628--683, 2016.

\bibitem[CS12]{ChelkakSmirnovIsing}
Dmitry Chelkak and Stanislav Smirnov.
\newblock Universality in the 2{D} {I}sing model and conformal invariance of
  fermionic observables.
\newblock {\em Invent. Math.}, 189(3):515--580, 2012.

\bibitem[DC13]{CDParafermionic}
Hugo Duminil-Copin.
\newblock Parafermionic observables and their applications to planar
  statistical physics models.
\newblock {\em Ensaios Matematicos}, 25:1--371, 2013.

\bibitem[Dub05]{DubedatSLEMart}
Julien Dub{\'e}dat.
\newblock {${\rm SLE}(\kappa,\rho)$} martingales and duality.
\newblock {\em Ann. Probab.}, 33(1):223--243, 2005.

\bibitem[Dub07]{DubedatCommutationSLE}
Julien Dub{\'e}dat.
\newblock Commutation relations for {S}chramm-{L}oewner evolutions.
\newblock {\em Comm. Pure Appl. Math.}, 60(12):1792--1847, 2007.

\bibitem[Dub09]{DubedatSLEDuality}
Julien Dub{\'e}dat.
\newblock Duality of {S}chramm-{L}oewner evolutions.
\newblock {\em Ann. Sci. \'Ec. Norm. Sup\'er. (4)}, 42(5):697--724, 2009.

\bibitem[HK13]{HonglerKytolaIsingFree}
Cl{\'e}ment Hongler and Kalle Kyt{\"o}l{\"a}.
\newblock Ising interfaces and free boundary conditions.
\newblock {\em Journal of the American Mathematical Society}, 26(4):1107--1189,
  2013.

\bibitem[Izy13]{IzyurovIsingMultiplyConnectedDomains}
Konstantin Izyurov.
\newblock Critical ising interfaces in multiply-connected domains, 2013.

\bibitem[Izy15]{IzyurovObservableFree}
Konstantin Izyurov.
\newblock Smirnov's observable for free boundary conditions, interfaces and
  crossing probabilities.
\newblock {\em Comm. Math. Phys.}, 337(1):225--252, 2015.

\bibitem[KL07]{KozdronLawlerMultipleSLEs}
Michael~J. Kozdron and Gregory~F. Lawler.
\newblock The configurational measure on mutually avoiding {SLE} paths.
\newblock In {\em Universality and renormalization}, volume~50 of {\em Fields
  Inst. Commun.}, pages 199--224. Amer. Math. Soc., Providence, RI, 2007.

\bibitem[KP16]{KytolaPeltolaMultipleSLE}
Kalle Kyt{\"o}l{\"a} and Eveliina Peltola.
\newblock Pure partition functions of multiple {SLE}s.
\newblock {\em Comm. Math. Phys.}, 346(1):237--292, 2016.

\bibitem[KS12]{KemppainenSmirnovRandomCurves}
Antti Kemppainen and Stanislav Smirnov.
\newblock Random curves, scaling limits and loewner evolutions.
\newblock {\em arXiv preprint arXiv:1212.6215}, 2012.

\bibitem[Law09]{LawlerPartitionFunctionsSLE}
Gregory~F. Lawler.
\newblock Partition functions, loop measure, and versions of {SLE}.
\newblock {\em J. Stat. Phys.}, 134(5-6):813--837, 2009.

\bibitem[LSW03]{LawlerSchrammWernerConformalRestriction}
Gregory~F. Lawler, Oded Schramm, and Wendelin Werner.
\newblock Conformal restriction: the chordal case.
\newblock {\em J. Amer. Math. Soc.}, 16(4):917--955 (electronic), 2003.

\bibitem[LW04]{LawlerWernerBrownianLoopsoup}
Gregory~F. Lawler and Wendelin Werner.
\newblock The {B}rownian loop soup.
\newblock {\em Probab. Theory Related Fields}, 128(4):565--588, 2004.

\bibitem[MS16a]{MillerSheffieldIG1}
Jason Miller and Scott Sheffield.
\newblock Imaginary geometry {I}: interacting {SLE}s.
\newblock {\em Probab. Theory Related Fields}, 164(3-4):553--705, 2016.

\bibitem[MS16b]{MillerSheffieldIG2}
Jason Miller and Scott Sheffield.
\newblock Imaginary geometry {II}: {R}eversibility of {$\operatorname{SLE}\sb
  \kappa(\rho\sb 1;\rho\sb 2)$} for {$\kappa\in(0,4)$}.
\newblock {\em Ann. Probab.}, 44(3):1647--1722, 2016.

\bibitem[MS16c]{MillerSheffieldIG3}
Jason Miller and Scott Sheffield.
\newblock Imaginary geometry {III}: reversibility of {$\rm SLE\sb \kappa$} for
  {$\kappa\in(4,8)$}.
\newblock {\em Ann. of Math. (2)}, 184(2):455--486, 2016.

\bibitem[Phe01]{Choquet}
Robert~R. Phelps.
\newblock {\em Lectures on {C}hoquet's theorem}, volume 1757 of {\em Lecture
  Notes in Mathematics}.
\newblock Springer-Verlag, Berlin, second edition, 2001.

\bibitem[Qia16]{QianConformalRestrictionTrichordal}
Wei Qian.
\newblock Conformal restriction: The trichordal case, 2016.

\bibitem[RS05]{RohdeSchrammSLEBasicProperty}
Steffen Rohde and Oded Schramm.
\newblock Basic properties of {SLE}.
\newblock {\em Ann. of Math. (2)}, 161(2):883--924, 2005.

\bibitem[SS13]{SchrammSheffieldContinuumGFF}
Oded Schramm and Scott Sheffield.
\newblock A contour line of the continuum {G}aussian free field.
\newblock {\em Probab. Theory Related Fields}, 157(1-2):47--80, 2013.

\bibitem[SW05]{SchrammWilsonSLECoordinatechanges}
Oded Schramm and David~B. Wilson.
\newblock S{LE} coordinate changes.
\newblock {\em New York J. Math.}, 11:659--669 (electronic), 2005.

\bibitem[WW13]{WernerWuCLEtoSLE}
Wendelin Werner and Hao Wu.
\newblock From {CLE}($\kappa$) to {SLE}($\kappa,\rho$).
\newblock {\em Electron. J. Probab.}, 18:no. 36, 1--20, 2013.

\bibitem[WW16]{WangWuLevellinesGFFI}
Menglu Wang and Hao Wu.
\newblock Level {L}ines of {G}aussian {F}ree {F}ield {I}: {Z}ero-{B}oundary
  {GFF}.
\newblock {\em Stochastic Processes and their Applications}, page to appear,
  2016.

\bibitem[Zha08]{ZhanReversibility}
Dapeng Zhan.
\newblock Reversibility of chordal {SLE}.
\newblock {\em Ann. Probab.}, 36(4):1472--1494, 2008.

\end{thebibliography}

\end{document}